%% file: Klein-Gordon_dynamics_above_the_threshold.tex
\documentclass[a4paper,12pt]{amsart}
\usepackage{import%,showkeys
}

\input{files/preamble.tex}
\input{files/macros.tex}

\input{files/environments.tex}

\newcommand{\floor}[1]{\lfloor #1 \rfloor}
\newcommand{\ceil}[1]{\left\lceil #1 \right\rceil}

\DeclareMathOperator{\eq}{eq}
\DeclareMathOperator{\linspan}{span}
\DeclareMathOperator{\dist}{dist}
\DeclareMathOperator{\rad}{rad}
\DeclareMathOperator{\inter}{interior}
\DeclareMathOperator{\tail}{tail}
\DeclareMathOperator{\var}{var}

\newcommand{\scatnorm}[1]{{\| #1 \|}_{L_t^3 L_x^6}}
\newcommand{\scnorm}[2]{{\| #1 \|}_{L_{t}^3(#2;L_x^6)}}

\newcommand{\embeds}{\hookrightarrow}
\newcommand{\LF}{\langle \grad_{x} \rangle^{-1} \grad_{x}}

\newcommand\restr[2]{{% we make the whole thing an ordinary symbol
  \left.\kern-\nulldelimiterspace % automatically resize the bar with \right
  #1 % the function
  \vphantom{\big|} % pretend it's a little taller at normal size
  \right|_{#2} % this is the delimiter
  }}

\title[Effective threshold dynamics]{Effective threshold dynamics for the cubic nonlinear Klein-Gordon equation in three dimensions}

\author{Jaime Gómez}
\address{}
\email{}
\author{Joachim Krieger}

\begin{document}
\begin{abstract}
    We consider the focussing cubic Klein-Gordon equation on $\mathbb{R}^{3+1}$ in the radial setting. Foundational work by Nakanishi-Schlag (\cite{NakanishiSchlagBook}) gives a complete 
    qualitative characterisation of the dynamics near the ground state $Q$, in particular revealing a threshold phenomenon, whereby a co-dimension one Lipschitz manifold separates finite time blow up from scattering solutions in forward time. Our aim in this paper is to derive an {\it{effective bound}} for the scattering norm $\|u\|_{L_t^3([0,\infty);L_x^6(\mathbb{R}^3))}$, for scattering energy class solutions $u$ with data sufficiently near $Q$ and close to the threshold manifold. We also derive effective bounds for scattering solutions in the Payne-Sattinger regime. This requires the development of an induction on energy strategy in a setting where an alternative to the usual Morawetz approach appears necessary. 
\end{abstract}

\maketitle

\vspace{0.7cm}

\section{Introduction}

Consider the cubic, focusing Klein-Gordon equation in $\R^{1+3}$, given by
\begin{equation}
    \label{eqn:NLKG}
    (\Box+1)u = u^3, \quad u[0] = (u_{0}, u_{1}),\quad \Box = \partial_t^2 - \triangle. 
\end{equation}
We shall restrict ourselves throughout to {\it{radial solutions}}, i. e. $u = u(t, x)$ only depending on $(t, |x|)$. The problem enjoys a conserved energy
\[ E(u, u_t) = \int_{\mathbb{R}^3}\frac12 |\nabla_{t,x}u|^2 + \frac12 u^2 - \frac14 u^4 \diff x ,\]
which defines the energy class $H^1 \times L^2$ of data with finite energy.

The well-posedness theory for this equation is concerned with the existence of strong solutions $u \in C([0,T);H^1) \times C^1([0,T);L^2)$ on a time interval $[0,T)$. Data belonging to the energy class lead to unique evolutions whose maximal forward time of existence $T$ can either be finite, in which case $\scnorm{u}{[0,T)} =\infty$ and the solution experiences a blowup, or infinite, meaning that $u$ exists for all times. In the latter case, if, moreover $\scatnorm{u} < \infty$, then $u$ scatters to a free wave in the sense that
\[ (u(t),u_t(t)) = (v(t),v_t(t)) + o_{H^1 \times L^2}(1) , \qquad t \to \infty .\]
The static equation
\[ - \triangle \varphi + \varphi = \varphi^3 \]
enjoys a unique radial, positive and stationary solution $Q$ \cite{Coffman, Gidas, McLeod, StraussSolitary}, which is furthermore $C^\infty$-smooth and decays exponentially as $|x|\rightarrow +\infty$. Crucially, $Q$ works as a critical threshold in the energy sense to establish a criterion by which finite-time blowup and existence for all times can be understood from the functionals
\[ K_0(u) = \int_{\R^3} \abs{\grad_x u}^2 + u^2 - u^4 \diff x , \quad K_2(u) = \int_{\R^3} \abs{\grad_x u}^2 - \frac{3}{4} u^4  \diff x \]
in the region $E(u,u_t) < J(Q)$, where $J$ is the static energy
\[ J(u) = \int_{\R^3} \frac{1}{2} \abs{\grad_x u}^2 + \frac{1}{2} u^2 - \frac{1}{4} u^4 \diff x ,\]
as proven by the works by Payne and Sattinger \cite{PayneSattinger}. Scattering to zero of solutions defined for all times was proven by Ibrahim, Masmoudi and Nakanishi \cite{IbrahimMasmoudiNakanishi} in the spirit of Kenig-Merle for the $H^1$-critical wave and Schrödinger equations \cite{KenigMerle}. In particular, they establish that solutions below the ground state energy share the sign of $K_0$ and $K_2$, which depends only on the data, and that blow-up in finite time occurs if $K_0(u_0) <0$, and scattering occurs when $K_0(u) > 0$, ensuring the existence of a function $N \colon(0,J(Q)) \longrightarrow (0,\infty)$ that satisfies
\begin{equation}\label{eqn:NfunctionQualitative}
    \scatnorm{u} < N(E(u,u_t))
\end{equation}
in the scattering regime. The cases where $K_0(u) = K_2(u) = 0$ are precluded by the variational characterization of $Q$, forcing that $u=Q$ in that case.

Moving slightly above the ground state energy, the characterization of the dynamics of solutions around $Q$ is of great interest, and a remarkable qualitative picture distinguishing $9$ different dynamical regimes was obtained in seminal work by Nakanishi-Schlag \cite{NakanishiSchlagBook}.

Limiting ourselves to the dynamics in {\it{forward time}} $t\geq 0$, the work by Nakanishi-Schlag in particular implies the existence of a co-dimension one Lipschitz manifold $\mathcal{M}$ near $(Q, 0)$ in the energy topology, and such that the following rough picture obtains:
\begin{itemize}
\item Energy class data $u[0]\in \mathcal{M}$ result in solutions scattering toward $Q$, i. e. the solution $\vec{u}(t,\cdot) = \big(u(t,\cdot), u_t(t,\cdot)\big)$ satisfies 
\[
\vec{u}(t,\cdot) = \big(Q + v,\,v_t\big)
\]
where $\big(v(t,\cdot), v_t(t,\cdot)\big)$ scatters toward a free Klein-Gordon solution as $t\rightarrow +\infty$, in the sense of $H^1 \times L^2$.  $\mathcal{M}$ is called the {\it{center stable manifold}} for the forward evolution. 
\item Energy class data $u[0]$ `below' $\mathcal{M}$, in a sense to be rendered precise below, result in solutions scattering toward $0$ as $t\rightarrow +\infty$. In particular, their {\it{scattering norm}} is bounded:
\begin{equation}\label{eq:scatteringbelowM}
\scnorm{u}{[0,\infty)} <+\infty.
\end{equation}
\item Energy class data $u[0]$ `above' $\mathcal{M}$ result in finite time blow up. 
\end{itemize}

It is intuitively clear that if we pick the data $u[0]$ in the second scenario closer and closer to the center stable manifold $\mathcal{M}$, the scattering norm will increase. Thus a natural question to ask is whether the qualitative statements \eqref{eqn:NfunctionQualitative} and \eqref{eq:scatteringbelowM} can be rendered quantitative in an effective manner. This is the question we pursue in the present paper.

The celebrated Kenig-Merle concentration-compactness approach fails to deliver an answer to such a question, as it hinges on a compactness method. The work of B. Dodson and J. Murphy \cite{DodsonMurphy} on the radial, focusing, cubic nonlinear Schrödinger equation establishes a first approach at proving scattering of solutions all while bypassing the Kenig-Merle argument. In particular, they build on the result of Holmer and Roudenko \cite{HolmerRoudenko}, who first obtained an ineffective estimate on the scattering norm of solutions \cite[Remark~1.3]{DodsonMurphy} below the energy threshold. Their approach instead relies on a scattering criterion by Tao \cite[Theorem~1.1]{TaoDPDE}, where quantitative control is lost due to a dependence on the initial profile $u[0]$.

While our approach hinges on certain tools also employed in \cite{DodsonMurphy}, such as the radial Strauss-Sobolev embedding and the virial identity \eqref{eqn:Virial}, we exploit finite speed of propagation and the existence of an explicit energy channel in order to establish an effective, recursive bound on the scattering norm of solutions. Moreover, we are not limited to those solutions below the ground state energy threshold.
In order to formulate our main result, we need to introduce some notions, which also clarify the trichotomy presented before.

Following \cite{NakanishiSchlagBook}, let $L_+ = -\triangle +1 - 3Q^2$ the operator arising upon linearising the time independent cubic Klein-Gordon equation around the ground state. This operator admits a unique negative eigenvalue $-k^2,\,k>0$, corresponding to an exponentially decaying, positive eigenfunction $\rho$, which we may assume to be normalised $\|\rho\|_{L^2(\mathbb{R}^3)} = 1$. The presence of this eigenvalue leads to generically exponential growth in forward time for the wave propagator associated to $L_+$, prevention of which requires imposition of a co-dimension one condition on the initial data. This subspace of the data corresponds to the tangent space to $\mathcal{M}$ at $(Q, 0)$. As the evolution of data $u[0]$ near $(Q, 0)$ is at first dominated by the wave propagation associated to $L_+$ until reaching a `fully nonlinear regime', it is natural to expect the norm $\|u\|_{L_t^3([0,\infty) L_x^6(\mathbb{R}^3)}$ to grow logarithmically in terms of the distance of $u[0]$ to $\mathcal{M}$. 

\begin{numthm}\label{thm:main} There exists an effectively computable constant $1\gg \eta_1>0$ with the following property: let $\vec{u}[0]\in B_{\eta_1}(Q,0)\cap \mathcal{M}$, and consider the data 
\[
\vec{\phi}[0]: = \vec{u}[0]-(a\rho,0),\,0<a<\eta_1.
\]
Then the solution of \eqref{eqn:NLKG} with data $\vec{\phi}[0]$ exists globally in forward time and scatters toward $0$. Furthermore, there are effectively computable constants $d_*, C_*>0$ such that 
\[
\big\|\phi\big\|_{L_t^3([0,\infty)L_x^6(\mathbb{R}^3)}\in \big[\log^{\frac13}(\frac{1}{a})\frac{\|Q\|_{L_x^6}}{k^{\frac13}} - d_*,\quad \log^{\frac13}(\frac{1}{a})\frac{\|Q\|_{L_x^6}}{k^{\frac13}} + C_*\big],
\]
In particular, we have 
\[
\lim_{a\rightarrow 0_+}\log^{-\frac13}(\frac{1}{a})\big\|\phi\big\|_{L_t^3([0,\infty)L_x^6(\mathbb{R}^3)} = \frac{\|Q\|_{L_x^6}}{k^{\frac13}}
\]
\end{numthm}
\begin{remark} The paper \cite{DuyckaertsMerle} reveals a similar limiting relation for solutions of the energy critical focussing nonlinear wave equation of energy level below that of the ground state and approaching the ground state. Our solutions are allowed to have energy above the ground state, and more importantly, our bounds are {\it{effective}}. 
\end{remark}
\begin{remark} The constant $d_*=d_*(\eta_1)$ and we have $\lim_{\eta_1\rightarrow 0}d_*(\eta_1) = 0$.The  constant $C_*$ resulting from our method is effectively computable. 
\end{remark}

%{\color{red} I am assuming the introduction with the definitions of the static energy $J$, the variational characterizations, the history of the problem and so on... that I will write once the rest is complete to have a good picture. And
%\[ {A}(s) = \sup_{E(u) \leq s} \scatnorm{u} .\]
%} 

The strategy to prove the theorem consists of two stages: in the following section, we first prove an energy induction step in the `Kenig-Merle dichotomy (or Payne-Sattinger) regime', i. e. for scattering solutions of energy $E(u,u_t)<J(Q)$. We can benefit here from the fact that we get coercivity for a certain functional appearing in a crucial virial identity purely on account of the `sub-criticality' of the energy level. This will allow us to show that scattering solutions are either in a perturbative regime to begin with, or have a tendency to decouple into weakly interacting pieces of smaller energy. 
\\
In section~\ref{sec:aboveQ}, we combine a perturbative argument describing the solution up to the `ejection stage' with the no-return result from \cite{NakanishiSchlagBook} to show that under the hypotheses of the theorem, the solution is either of perturbative type, or else decomposes into two weakly interacting solutions in the sub-critical regime.

\section{Solutions below the ground state energy}\label{sec:belowQ}

Throughout this section, we shall consider (radial) finite energy solutions of \eqref{eqn:NLKG} which are scattering. Equivalently, we require that $K_0(u)\geq 0$. For $s<J(Q)$, we then introduce the quantity 
\[
A(s) = \sup_{E(u)\leq s,\,K_0(u)\geq 0}\scatnorm{u},
\]
where the norm is taken globally in time. As seen in \cite{PayneSattinger, NakanishiSchlagBook, IbrahimMasmoudiNakanishi} this expression is finite. The following proposition will give us a tool to effectively bound this quantity. 
\begin{numprop}
    \label{prop:ScatteringBound}
    Let $u$ be a radial energy class solution to \eqref{eqn:NLKG} of scattering type, satisfying  
    \[ E(u) = \int_{\mathbb{R}^3}\frac12 |\nabla_{t,x}u|^2 + \frac12 u^2 - \frac14 u^4 \diff x < J(Q) - \delta \]
    for some $\delta>0$. Then there is an effectively computable $\delta' = \delta'(\delta)>0$ such that  we have 
    \begin{align*}
        \scatnorm{u} \leq K\Big( A\big(J(Q) - \delta -\delta' \big)\Big),
    \end{align*}
    for an effectively computable growing function $K: \mathbb{R}_+\rightarrow \mathbb{R}_+$. 
    In particular
    \begin{align*}
        A\big(J(Q) - \delta \big)\leq K\Big( A\big(J(Q) - \delta -\delta' \big)\Big).
    \end{align*}
    The function $K$ can be chosen of the form $K(r) = C^{C^{C^{cr^3}}}$ for suitable positive constants $C,c$. 
\end{numprop}

The strategy behind Proposition~\ref{prop:ScatteringBound} proceeds in a number of steps. We first make a distinction between the following two cases. If $u$ never gathers a significant amount of $H_x^1$--norm near the origin, then we may bootstrap the smallness of its Strichartz norms. If this is not the case, then we argue that either there is an explicit channel through which $u$ disperses its energy far away from the origin, or it must gather sufficient energy far away from the origin to begin with. In turn, this provides a decoupling of the solution into two separate scattering pieces, allowing us to estimate $\scatnorm{u}$ in terms of the scattering norm of lower-energy solutions.

\section{Bootstrap argument under the small energy assumption}

We begin with a bootstrap argument that provides an explicit estimate on $\scatnorm{u}$ whenever $u$ gather its $H_{x}^1$--concentration far away from the origin at all times.

\begin{numremark}
    From this point on, we shall always refer to balls centered at the origin, and we will simply write $B_{R}$ instead of $B_{R}(0)$. Additionally, at every step in every proof where an absolute, computable constant $C>0$ appears, we shall let it absorb other absolute, positive computable constants. For instance, we may substitute $2C$ or $C^2$ by $C$ from one line to the next.  
\end{numremark}

The first part of the argument will require localized $L_{x}^p$ estimates on radial, finite energy solutions, both inside and outside balls centered at the origin. 

\begin{numlemma}
    \label{lemma:InterpolationBounds}
    Let $u$ be a radial solution to \eqref{eqn:NLKG} that exists globally in forward time and satisfies $E(u,u_{t}) < J(Q)$, and let $R>1$. Then it holds that
    \begin{equation}
        \label{eqn:InterpolationStrauss}
        \norm{u}_{L_{x}^p(|x| \geq R)} \leq C E(u)^{\frac{1}{2}} R^{-\frac{p-2}{p}} ,
    \end{equation}
    and
    \begin{equation}
        \label{eqn:SobolevBall}
        \norm{u}_{L^6(B_{R})} \leq C \left[ R^{-2} \norm{u}_{L^2(B_{R})}^2 + \norm{\grad u}_{L^2(B_{R})}^2 \right]^{1/2} .
    \end{equation}
\end{numlemma}
\begin{proof}
    The first part follows by an interpolation argument after an application of the Strauss radial lemma \cite[Radial Lemma~1]{StraussRadial}. In fact, under the ground state energy, the works of Payne and Sattinger \cite{PayneSattinger} establish the coercivity $K_{0},K_{2} >0$, and assuming that $E(u,u_{t}) < J(Q)-\delta$, then we find
    \[ J(Q)-\delta > E(u) > \int_{\mathbb{R}^n} \frac{1}{2}\lvert u_{t} \rvert^2+ \frac{1}{4}(\lvert u \rvert^2+\lvert \nabla u \rvert^2) \, \mathrm{d}x \geq\frac{1}{4} \lVert u \rVert_{H^1}^2 ,\]
    from where a uniform $L_{x}^2$ estimate is deduced. On the other hand, the Strauss radial lemma provides
    \[ \lvert u(t,x) \rvert \leq C_{0} \lvert x \rvert^{-1} \lVert u(t,\cdot) \rVert_{H^1} \leq C_{1} \lvert x \rvert^{-1}E(u)^{\frac12} \]
    uniformly in $t$ and for $|x| > 1$. By interpolating both estimates, we conclude that
    \begin{equation*}
        \norm{u}_{L_{x}^p(|x| \geq R)} \leq \norm{u}_{L_{x}^2 (|x| \geq R)}^{\frac{2}{p}} \norm{u}_{L_{x}^{\infty}(|x| \geq R)}^{\frac{p-2}{p}} \leq C_{0} E(u)^{\frac{1}{2}} R^{-\frac{p-2}{p}} ,
    \end{equation*}
    for all $p>2$ and $R>1$.

    Equation~\ref{eqn:SobolevBall} reflects the Sobolev inequality on the ball $B_{R}$ present in \cite[{Theorem~6}, \S~5.6.3]{Evans} after a scaling argument is used to explicitly estimate the Sobolev constant. Indeed, we have
    \[ \norm{u}_{L^6(B_{R})} \leq C_{R} \norm{u}_{H^1(B_{R})} ,\]
    for a constant $C_{R} > 0$ depending only on $R$. By defining
    \[ v(y) = u(Ry) , \quad y \in B_{1} ,\]
    one may prompty check that
    \[ \norm{v}_{L^6(B_{1})} = R^{-1/2} \norm{u}_{L^6(B_{R})} , \quad \norm{v}_{H^1(B_{1})}^2 = R^{-3} \norm{u}_{L^2(B_{R})}^2 + R^{-1} \norm{\grad u}_{L^2(B_{R})}^2 .\]
    Applying \cite[{Theorem~6}, \S~5.6.3]{Evans} to $v$ on $B_{1}$ we conclude \eqref{eqn:SobolevBall}, where $C=C_{1}$.  
\end{proof}

By exploiting the estimates in Lemma~\ref{lemma:InterpolationBounds} together with the Strichartz estimates in \cite{NakanishiSchlagBook}, we may determine explicitly how much $u$ is allowed concentrate --and exactly how far away from the origin-- in the $H_x^1$--norm in order for a bootstrap scheme to yield global scattering estimates.

\begin{numlemma}
    \label{lemma:Bootstrap}
    Let $u$ be a globally defined, radial solution to \eqref{eqn:NLKG} satisfying $E(u, \partial_{t} u) < J(Q)$. 
    There are explicitly computable $\delta_*>0, R_*>0$ such that if
    \begin{equation}
        \label{eqn:H1Bootstrap}
        \|u \|_{L_t^\infty H_{x}^1 (B_{R_*})} < \delta_*,
    \end{equation}
    then it holds that
    \[ \|u\|_{L_{t}^3 L_{x}^6} \leq 2 \norm{u[0]}_{H^1 \times L^2} .\]
\end{numlemma}
\begin{proof}
    Notice that we have
    \begin{equation}
        \label{eqn:InterpolationSmallness}
        \norm{u}_{L_{t}^{\infty} L_{x}^{q}(|x| < R_*)} \leq C \delta_{*} , \quad q \in [2,6]
    \end{equation}
    by assuming $R_*>1$, applying \eqref{eqn:SobolevBall} to \eqref{eqn:H1Bootstrap} and interpolating the resulting $L_{x}^6(B_{R_*})$ estimate with the $L_{x}^2(B_{R_*})$ bound in \eqref{eqn:H1Bootstrap} itself.
    The Strichartz estimates in \cite[Lemma~2.46]{NakanishiSchlagBook} ensure that
    \[ \norm{u}_{L_{t}^2 B_{6,2}^{1/6} \cap L_{t}^{\infty} H_{x}^1} \lesssim \norm{u[0]}_{H^1 \times L^2} + \|u^3\|_{L_{t}^2 B_{6/5,2}^{5/6} + L_{t}^1 L_{x}^2} \]
    which we may use together with the embedding $L_{t}^2 B_{6,2}^{1/6} \embeds L_{t}^2 L_{x}^{9}$ to bound
    \begin{align*}
        \norm{u}_{L_{t}^2 L_{x}^9}
        &\lesssim \norm{u[0]}_{H^1 \times L^2} + \| u^3 \|_{L_{t}^1 L_{x}^2}
        \lesssim \norm{u[0]}_{H^1 \times L^2} + \norm{ \| u \|_{L_{x}^{18/5}} \|u\|_{L_{x}^9}^2 }_{L_{t}^1} \\
        &\lesssim \norm{u[0]}_{H^1 \times L^2} + \of{\norm{u}_{L_{t}^{\infty} L_{x}^{18/5}(|x|<R)} + \norm{u}_{L_{t}^{\infty} L_{x}^{18/5}(|x| \geq R)} } \norm{u}_{L_{t}^{2}L_{x}^{9}}^2 .
    \end{align*}
    At this point we may estimate both terms inside the parentheses by using \eqref{eqn:InterpolationStrauss} and \eqref{eqn:InterpolationSmallness} since $2 < 18/5 < 6$. This yields
    \[ \norm{u}_{L_{t}^2 L_{x}^9} \lesssim \norm{u[0]}_{H^1 \times L^2} + \of{C \delta_{*} + C E(u)^{\frac{7}{9}} R^{-\frac{4}{9}}} \norm{u}_{L_{t}^{2}L_{x}^{9}}^2 .\]
    It suffices to run a bootstrap/continuity argument now, which concludes by letting
    \[ C \delta_{*} + C E(u)^{\frac{7}{9}} R_{*}^{-\frac{4}{9}} \leq (4 \norm{u[0]}_{H^1 \times L^2})^{-1} .\]
    The final estimate we reach is $\|u\|_{L_{t}^2 L_{x}^9} \leq 2 \|u[0]\|_{H^1 \times L^2}$, and using the Strichartz estimates in \cite[Lemma~2.43]{NakanishiSchlagBook} one may obtain the same estimate for the $L_t^3 L_{x}^6$--norm.  
\end{proof}

\section{Energy channel}

In the following, we will work in a suitably localized setting in order to exploit the virial identity in \cite{NakanishiSchlagBook}. We shall consider $R>0$ and $\chi \in C_{c}^{\infty}(\R^3)$ a smooth, radial function that equals $1$ on $|x|<1$ and $0$ whenever $|x| > 2$, and satisfies $\| \chi \|_{C^1} \leq C_{\chi}$, and localize $u$ to $|x| \leq R + t$ by introducing
\begin{equation}
    \label{eqn:DefinitionTildeu}
    \Tilde{u}_{R}(t,x) = \zeta_{R}(t,x) \cdot u(t,x) , \quad \zeta_{R}(t,x) = \chi \of{\frac{x}{R + t}},\,t\geq 0.
\end{equation}
We shall, for now, not fix our choice of $R$, and instead we exhibit certain properties of $\Tilde{u}_R$ provided some mild conditions on how large $R$ must be taken. In particular, we will fix $R$ appropriately in terms of $R_{*}$ later on, but to this end we first need to make sure that we can recover the coercivity hypotheses for $\Tilde{u}_R$. The first step towards this will be the following version of finite speed of propagation.

\begin{numlemma}
    \label{lemma:FiniteSpeedOfPropagation}
    Let $u$ be a solution to \eqref{eqn:NLKG} that is globally defined in forward time and satisfies $E(u, u_{t}) < J(Q)$, and let $R > 0$. Then there is $R_{0}>0$ such that
    \[ E(u[0]|_{B_{R}}) \leq E(u[t] |_{B_{R+t}}) \]
    for all $t \geq 0$ provided that $R \geq R_{0}$.
\end{numlemma}
\begin{proof}
    Multiplying \eqref{eqn:NLKG} by $u_{t}$ yields the vanishing divergence condition
    \[ 0 = \divergence_{t,x}\left( \frac{1}{2} \left[ u_{t}^2 + u^2 + \lvert \nabla u \rvert^2 - \frac{1}{2} u^4 \right], -u_{t} \nabla u \right) .\]
    Integrating this over a forwards cone $K_{[0,t]}$ with base $B_{R}$, lid $B_{R+t}$ and mantle $M_{[0,t]}$ then gives
    \[ E(u[0] |_{B_{R}}) = E(u[t] |_{B_{R+t}}) + \frac{1}{\sqrt{2}} \int_{M_{[0,t]}} (-1, \omega) \cdot \left( \frac{1}{2} \left[ u_{t}^2 + u^2 + \lvert \nabla u \rvert^2 - \frac{1}{2} u^4 \right], -u_{t} \nabla u \right) \diff S ,\]
    where $(t,x) = (t,|x| \omega)$. We may rewrite
    \begin{align*}
        \int_{M_{[0,t]}} & (-1, \omega) \cdot \left( \frac{1}{2} \left[ u_{t}^2 + u^2 + \lvert \nabla u \rvert^2 - \frac{1}{2} u^4 \right], -u_{t} \nabla u \right) \diff S \\
        &= \int_{M_{[0,t]}} \frac{1}{2}u^2 \of{\frac{1}{2} u^2 - 1}  - \frac{1}{2} \of{(u_{t} + \grad u \cdot \omega)^2 + \abs{\grad u}^2 - (\grad u \cdot \omega)^2} \diff S ,
    \end{align*}
    and notice that the first expression in parentheses is at most zero if $R > R_{0} = (CJ(Q))^{-1}$ (where $C$ is given by the Strauss radial lemma) and the second one is at most zero thanks to the Cauchy-Schwarz inequality.
\end{proof}

\begin{numlemma}
    \label{lemma:Localizedu}
    Let $u$ be a solution to \eqref{eqn:NLKG} that is globally defined in forward time and satisfies $E(u, u_{t})<J(Q)-\delta$, and for $R \geq R_{*}$, consider $\Tilde{u}_{R}$ as in \eqref{eqn:DefinitionTildeu}. Then there is $\Tilde{R} > 0$ such that for all $R > \Tilde{R}$, 
    \[ E(\Tilde{u}_{R}[t]) \leq J(Q)-\frac{\delta}{2} , \quad K_{j}(\Tilde{u}_{R}(t)) > 0 , \quad j=0,2 , \ t \geq 0. \]
    Moreover, $\Tilde{u}_R$ satisfies the equation
    \[ (\Box+1) \Tilde{u}_{R} = \Tilde{u}_R^3 + \Tilde{E}_{R} ,\]
    where
    \begin{equation}
        \label{eqn:ErrorTerm}
        \Tilde{E}_{R} = (\zeta_{tt} - \Delta \zeta) u + 2 \zeta_{t} u_{t} - 2 \grad \zeta \cdot \grad u + \zeta(1-\zeta^2) u^3 , \quad \zeta(t,x) = \chi\of{\frac{x}{R+t}} .
    \end{equation}
    and
    \[ \Tilde{R} = \max \inset{\frac{3}{2} , \frac{J(Q)}{\delta} \left[ \frac{1+8CJ(Q) C_{\chi}^2}{4CJ(Q)} +2C_{\chi} \right]} .\]
    
\end{numlemma}
\begin{proof}
    We begin by ensuring the bound on the energy of $\Tilde{u}_{\Tilde{R}}$.
    Start by estimating
    \begin{align*}
        E(\Tilde{u}_{R}[t]) &= E(u[t]) + E(\Tilde{u}[t]|_{B_{R+t}^c}) - E(u[t]|_{B_{R+2t}^c}) - E(u[t]|_{B_{R+t}^c \cap B_{R+2t}}) \\
        &\leq E(u[t]) + E(\Tilde{u}[t]|_{B_{R+t}^c}) - E(u[t]|_{B_{R+t}^c \cap B_{R+2t}})
    \end{align*}
    since the energy density of $u(t)$ is nonnegative for $|x|>R$ if $R > (CJ(Q))^{-1}$ by the Strauss radial lemma. We can now write
    \begin{align*}
        E(\Tilde{u}_{R}&|_{B_{R+t}^c}) - E(u|_{B_{R+t}^c \cap B_{R+2t}}) \\
        &= \int_{B_{R+t}^c \cap B_{R+2t}} \frac{1}{2} \left[
        \abs{\Tilde{u}_{R}}^2 - \abs{u}^2 + \abs{\grad_{x} \Tilde{u}_{R}}^2 - \abs{\grad_{x} u}^2 + \abs{\partial_{t}\Tilde{u}_{R}}^2 - \abs{u_{t}}^2 \right]
        -\frac{1}{4} \left[ \abs{\Tilde{u}_{R}}^4 - \abs{u}^4   \right]  \diff x \\
        &= \int_{B_{R+t}^c \cap B_{R+2t}}
        \frac{\zeta_{R}^2-1}{2}|u|^2 + \frac{1-\zeta_{R}^4}{4}|u|^4
        +\frac{\zeta_{R}^2 -1}{2} |\grad_{x}u|^2 + \frac{\zeta_{R}^2 -1}{2} |u_{t}|^2 \\
        &\hspace{3cm} + \frac{|u \grad_{x} \zeta_{R}|^2}{2} + \frac{|u \partial_{t} \zeta_{R}|^2}{2} + u \zeta_{R} \grad_{x} u \cdot \grad_{x} \zeta_{R} + u \zeta_{R} u_{t} \partial_{t} \zeta_{R}
        \diff x \\
        &\leq \int_{B_{R+t}^c \cap B_{R+2t}}
        \left[
            \frac{1-\zeta_{R}^4}{4C J(Q) (R+t)^2} + \frac{\abs{\grad_{x} \chi \of{\frac{x}{R+t}}}^2}{2(R+t)^2} + \frac{\abs{x \cdot \grad_{x} \chi \of{\frac{x}{R+t}}}^2}{2(R+t)^4}
        \right] |u|^2 \\
        &\hspace{3cm} + \zeta_{R} \frac{|u|^2 + |\grad_{x} u|^2}{2} \cdot \frac{\abs{\grad_{x} \chi \of{\frac{x}{R+t}}}}{R+t} + \zeta_{R} \frac{|u|^2 + |u_{t}|^2}{2} \cdot \frac{\abs{x \cdot \grad_{x} \chi \of{\frac{x}{R+t}}}}{(R+t)^2}
        \diff x \\
        &\leq \left[ 
            \frac{1+8CJ(Q) C_{\chi}^2}{4CJ(Q)(R+t)^2} + \frac{2C_{\chi}}{R+t}
        \right]
        \int_{B_{R+t}^c \cap B_{R+2t}} |u|^2 + |\grad_{x} u|^2 + |u_{t}|^2 \diff x  \\
        &\leq \frac{J(Q)}{R+t} \left[ \frac{1+8CJ(Q)C_{\chi}^2}{4CJ(Q)(R+t)} + 2C_{\chi} \right]
        \leq \frac{J(Q)}{R} \left[ \frac{1+8CJ(Q) C_{\chi}^2}{4CJ(Q)} +2C_{\chi} \right]
        \leq \frac{\delta}{2} ,
    \end{align*}
    where $R$ is chosen explicitly in terms of $\delta$ to satisfy the last estimate. We conclude that
    \begin{equation}
        \label{eqn:TildeuEnergyBelowGroundState}
        E(\Tilde{u}_R[t]) \leq E(u[t])+\frac{\delta}{2} \leq J(Q) - \frac{\delta}{2} .
    \end{equation}
    Moving on to the functionals $K_{0}$ and $K_{2}$, we can estimate
    \begin{align*}
        \abs{K_{0}(u)-K_{0}(v)}
        &\leq \abs{\norm{u}_{H^1}^2 - \norm{v}_{H^1}^2} + \abs{\norm{u}_{L^4}^4 - \norm{v}_{L^4}^4} \\
        &\lesssim \of{\norm{u}_{H^1} + \norm{u}_{H^1}^3 + \norm{v}_{H^1} + \norm{v}_{H^1}^3} \norm{u-v}_{H^1} 
    \end{align*}
    and
    \begin{align*}
        \abs{K_{2}(u)- K_{2}(v)}
        &\leq \abs{\norm{\grad u}_{L^2}^2 - \norm{\grad v}_{L^2}^2} + \frac{3}{4} \abs{\norm{u}_{L^4}^4 - \norm{v}_{L^4}^4} \\
        &\lesssim \of{\norm{\grad u}_{L^2} + \norm{\grad v}_{L^2} + \norm{u}_{H^1}^3 + \norm{v}_{H^1}^3} \norm{u - v}_{H^1} ,
    \end{align*}
    where we have estimated the $L^4$-norms in terms of the $H^1$-norms by interpolating $L^2$ and $L^6$ and embedding $H^1 \embeds L^6$.
    At the same time, we recover the positivity of $K_{0}(\Tilde{u}_{R})$ and $K_{2}(\Tilde{u}_{R})$ for a qualitatively large $R$ via a contradiction argument. Indeed, Lemma~2.9 and Lemma~2.21 in \cite{NakanishiSchlagBook} ensure that if $K_{j}(\Tilde{u}_{R}(t)) < 0$ for $j=0,2$, then
    \[ K_{j}(\Tilde{u}_{R}(t)) < -2 (E(Q)-E(\Tilde{u}_{R})) \leq -\delta .\]
    However, we can estimate
    \[ \norm{u(t)-\Tilde{u}_{R}(t)}_{H^1}
    \lesssim \of{1+\frac{1}{R}}\norm{u(t)}_{H^1(B_{R +t}^c)} \lesssim \of{1+\frac{1}{R}} E(u[0]|_{B_{R}^c}) \]
    by Lemma~\ref{lemma:FiniteSpeedOfPropagation}, and by letting $R > R_{2}$ for some large enough $R_{2}$ (depending only on the energy distribution of the data), this becomes at most $c\delta_{2}$ for any $0<c<1$. We conclude that
    \[ K_{j}(u) \leq -\frac{\delta}{2} ,\]
    which contradicts $K_{j}(u) > 0$.

    In order to recover the positivity for $\Tilde{R}$, we note that for any fixed time, $K_{j}(\Tilde{u}_{R}(t))$ is a continuous function of $R$. Assuming that there is $t > 0$ for which $K_{j}(\Tilde{u}_{\Tilde{R}}(t)) \leq 0$ ensures that there is $\Tilde{R} \leq R < R_{2}$ such that $K_{j}(\Tilde{u}_R(t)) = 0$. But then $J(\Tilde{u}_R(t)) \geq J(Q)$ by Lemma~2.9 and Lemma~2.21 in \cite{NakanishiSchlagBook}, which contradicts \eqref{eqn:TildeuEnergyBelowGroundState}. Hence $K_{j}(\Tilde{u}_{\Tilde{R}}(t)) > 0$.

    Finally, that $\Tilde{E}_R$ obeys \eqref{eqn:ErrorTerm} follows from computing $(\Box+1) (\zeta u)$ and using \eqref{eqn:NLKG}.
\end{proof}

The previous proposition establishes that we may always recover the energy and coercivity hypotheses of $u$ when we restrict to a sufficiently large, explicit cutoff.
From this point on, we shall write $\Tilde{u}$ instead of $\Tilde{u}_R$.

In order to expose a channel through which the free energy escapes away from the origin, we will need to exploit the virial identity in \cite{NakanishiSchlagBook} together with the lower bound on $K_{2}(\Tilde{u})$ from the variational analysis in \cite[Lemma~2.21]{NakanishiSchlagBook} in an appropriate averaged sense. To this end, we shall require the following estimate on the spreading behavior of a solution.

\begin{numlemma}\label{lemma:LowerBoundGradient}
    Let $u$ be a globally defined solution in forward time to \eqref{eqn:NLKG} satisfying $E(u, u_{t})<J(Q)$, and assume that $\|u[0]\|_{H_{x}^{1}(B_{R_{*}})}  \geq \delta_{*}$ for $\delta_*, R_{*}$ as in Lemma~\ref{lemma:Bootstrap}. Furthermore, let $R>\max\{R_*, \tilde{R}\}$ as in the preceding lemma, and set $\tilde{u} = \tilde{u}_R$. Then there is $\delta_{1} = \delta_{1}(\delta_{*},R_{*}) >0$ such that
    \[ \int_{I} \norm{\grad_{x} \Tilde{u}}_{L_{x}^2}^2 \diff t \geq \delta_{1},\,\tilde{u} = \tilde{u}_R,\,R\geq R_*, \]
    for any interval $I \subset [0,\infty)$ of length $|I|=1$. Moreover, there is an absolute constant $C > 0$ such that $\delta_{1}$ can be taken explicitly at most $C (\delta_{*} / R_{*})^{16}$
\end{numlemma}
\begin{proof}
    We assume that there is $I \subset \R_+$ with $|I| = 1$ such that
    \begin{equation}
        \label{eqn:SmallGradient}
        \int_{I} \norm{\grad_{x} \Tilde{u}}_{L_{x}^2}^2  \diff x < \delta_{1}
    \end{equation}
    for $\delta_{1} > 0$. The argument will hinge on the low-frequency concentration of the data $\Tilde{u}[T_*]$ for a certain $T_{*} \in I$ which we will exploit together with an improved Strichartz estimate in the radial case by Guo and Wang \cite{GuoRadialStrichartz} to force the smallness of the free energy of $u$ near the origin, which contradicts the hypotheses.
    Define
    \[ G_{1} \coloneqq \inset{t \in I \colon \norm{\grad_{x} \Tilde{u}}_{L_{x}^2}^2 < 5 \delta_{1}}. \]
    Then $|G_{1}| \geq 3/4$, as otherwise 
    \[ \int_{I} \norm{\grad_{x} \Tilde{u}}_{L_{x}^2}^2 \diff t \geq \int_{I\cap G_{1}^c} \norm{\grad_{x} \Tilde{u}}_{L_{x}^2}^2  \diff t \geq 5 \cdot \frac{1}{4} \delta_{1} > \delta_{1} ,\]
    which contradicts \eqref{eqn:SmallGradient}.
    We need to refine this so that
    \[ \norm{\LF \Tilde{u}_{t}}_{L_{x}^2}^2 \lesssim \delta_{1} \]
    for some $t \in G_{1}$. 
    To this end, notice that by differentiating $\langle \LF \Tilde{u}_{t}, \LF \Tilde{u} \rangle$ with respect to $t$ and the integrating over $t \in [t_{1},t_{2}] \subset I$, we find   
    \begin{align*}
        \int_{t_{1}}^{t_{2}} &\norm{\LF \Tilde{u}_{t}}_{L_{x}^2}^2 \diff t \\
        &= \int_{t_{1}}^{t_{2}} \norm{\LF \Tilde{u}}_{L_{x}^2}^2 \diff t
        - \int_{t_{1}}^{t_{2}} \langle \LF \Delta \Tilde{u} , \LF \Tilde{u} \rangle  \diff t \\
        &\hspace{0.5cm}- \int_{t_{1}}^{t_{2}} \langle \LF \Tilde{u}^3, \LF \Tilde{u} \rangle \diff t
        - \int_{t_{1}}^{t_{2}} \langle \LF E , \LF \Tilde{u} \rangle \diff t \\
        &\hspace{0.5cm}+ \langle \LF \Tilde{u}_{t}, \LF \Tilde{u} \rangle\Big|_{t_{1}}^{t_{2}} .
    \end{align*}
    We may choose $t_{1},t_{2} \in G_{1}$ satisfying $t_{2}-t_{1} \geq 3/4$, and estimate every term above. Since $\LF$ can be estimated by its Fourier multiplier $|\xi|\langle \xi \rangle^{-1} \leq 1$ in $L_{x}^2$, and since $\langle \grad_{x} \rangle^{-1}$ corresponds to $\langle \xi \rangle^{-1} \leq 1$, we may remove these adequately in the terms above, together with an application of $K_{2}(\Tilde{u}) > 0$, to reach
    \begin{align*}
        \int_{t_{1}}^{t_{2}} &\norm{\LF \Tilde{u}_{t}}_{L_{x}^2}^2 \diff t \\
        &\leq 3\int_{t_{1}}^{t_{2}} \norm{\grad_{x} \Tilde{u}}_{L_{x}^2}^2 \diff t
        + \int_{t_{1}}^{t_{2}} \norm{\grad_{x} \Tilde{u}}_{L_{x}^2} \norm{E}_{L_{x}^2} \diff t \\
        &\hspace{3cm}+ \norm{\Tilde{u}_t(t_{2})}_{L_{x}^2} \norm{\grad_{x}\Tilde{u}(t_{2})}_{L_{x}^2}
        + \norm{\Tilde{u}_t (t_{1})}_{L_{x}^2} \norm{\grad_{x}\Tilde{u}(t_{1})}_{L_{x}^2} \\
        &\leq  3 \int_{t_{1}}^{t_{2}} \norm{\grad_{x} \Tilde{u}}_{L_{x}^2}^2 \diff t
        + C \of{\int_{t_{1}}^{t_{2}} \norm{\grad_{x} \Tilde{u}}_{L_{x}^2}^2 \diff t}^{1/2}\\
        &\hspace{3cm} + \norm{\Tilde{u}_t(t_{2})}_{L_{x}^2} \norm{\grad_{x}\Tilde{u}(t_{2})}_{L_{x}^2}
        + \norm{\Tilde{u}_t (t_{1})}_{L_{x}^2} \norm{\grad_{x}\Tilde{u}(t_{1})}_{L_{x}^2} \leq C \delta_{1}^{1/2} .
    \end{align*}
    We may argue as before to show that for a subset $G_{2} \subset [t_{1}, t_{2}]$ of measure at least $(3/4) \cdot |t_{2}-t_{1}|$ we must have $\norm{\LF \Tilde{u}_{t}}_{L_{x}^2}^2 \leq c \delta_{1}^{1/2}$ for some absolute constant $c > 0$,
    and therefore choose $T_{*} \in G_{1} \cap G_{2}$, as if this intersection were empty, then
    \[ 1=|I| \geq |G_{1}| + |G_{2}| \geq \frac{3}{4} + \frac{3^2}{4^2} > 1 .\]
    Hence
    \[ \norm{\grad_{x} \Tilde{u}(T_{*})}_{L_{x}^2} \leq C \delta_{1}^{\frac{1}{2}} , \quad \norm{\LF \Tilde{u}_t (T_{*})}_{L_{x}^2} \leq C \delta_{1}^{\frac{1}{4}} .\]
    At this point we exploit the improved radial Strichartz estimates in \cite{GuoRadialStrichartz}, where if
    \[ K(t) = \omega^{-1} \sin(t \omega), \quad K'(t) = \cos(t \omega), \quad K_{\pm}(t) = e^{\pm it\sqrt{1-\triangle}},\quad \omega =\sqrt{I-\Delta} ,\]
    then
    \begin{equation}\label{eq:Strichartzlocalised} \norm{K_{\pm}(t)P_{k} f}_{L_{t,x}^3 (\R^{3+1})} \leq C(3,k) \norm{f}_{L_{x}^2(\R^3)}
    \end{equation}
    where $P_{k}$ are the Littlewood-Paley projectors and
    \[
    C(3,k) =
        \begin{cases}
            2^{-\frac{k}{6}} \quad k \leq 0 \\
            2^{\frac{k}{6}} \quad k \geq 0 .
        \end{cases}        
    \]
    Now, if $f$ is supported on low frequencies $\abs{\xi} \sim 2^{k}$, $k \leq 0$, then the previous in combination with Bernstein's inequality grants the decay estimate
    \begin{align*}
        \norm{K(t) P_{k} f}_{L_{t}^3 L_{x}^6} \lesssim \left(2^{k \cdot \left(\frac{1}{3} - \frac{1}{6}  \right)}\right)^3 \norm{K(t) P_{k} f}_{L_{t,x}^3} \lesssim 2^{\frac{k}{2} - \frac{k}{6}} \norm{f}_{L_{x}^2} = 2^{\frac{k}{3}} \norm{f}_{L_{x}^2}
    \end{align*}

    We will use this estimate for low frequencies in the data, and the regular Strichartz estimate from \cite{NakanishiSchlagBook} for high frequencies. Consider a solution $\ell$ to the \emph{linear} Klein-Gordon equation with data $\ell[0]=(u_{0},u_{1}) = (\Tilde{u}(T_{*}), \Tilde{u}_t(T_{*}))$,
    \[ \ell = K'(t) u_{0} + K(t) u_{1} = \sum_{k \in \Z} K'(t) P_{k} u_{0} + \sum_{k \in \Z} K(t) P_{k} u_{1} .\]
    We estimate its $L_{t}^3 L_{x}^6$ norm as follows. Assuming $\kappa \leq 0$, 
    \begin{align*}
        \norm{\ell}_{L_{t}^3 L_{x}^6} &\leq
        \sum_{k = -\infty}^{\kappa} \norm{K'(t) P_{k} u_{0}}_{L_{t}^3 L_{x}^6}
        + \sum_{k = -\infty}^{\kappa} \norm{K(t) P_{k} u_{1}}_{L_{t}^3 L_{x}^6} \\
        &\hspace{1cm} +  \norm{K'(t) P_{\geq \kappa} u_{0}}_{L_{t}^3 L_{x}^6}
        +  \norm{K(t) P_{\geq \kappa} u_{1}}_{L_{t}^3 L_{x}^6} \\
        &\leq \sum_{k=-\infty}^{\kappa} 2^{\frac{k}{3}} \norm{u_{0}}_{H_{x}^1}
        + \sum_{k=-\infty}^{\kappa} 2^{\frac{k}{3}} \norm{u_{1}}_{L_{x}^2}
        + \norm{P_{\geq \kappa} u_{0}}_{H_{x}^1} 
        + \norm{P_{\geq \kappa} u_{1}}_{L_{x}^2}
    \end{align*}
    We estimate the last two terms as follows.
    \begin{align*}
        \norm{P_{\geq \kappa} u_{0}}_{H_{x}^1}^2
        &= \norm{\fourier{u_{0}}(\xi)}_{L^2({\abs{\xi} > 2^\kappa})}^2
        + \norm{\xi \fourier{u_{0}}(\xi)}_{L^2({\abs{\xi} > 2^\kappa})}^2
        \leq (1+2^{-2\kappa}) \norm{\xi \fourier{u_{0}}(\xi)}_{L_{\xi}^2}^2
        \leq (1+2^{-2\kappa}) \delta_{1} . \\
        \norm{P_{\geq \kappa} u_{1}}_{L_{x}^2}
        &\leq \norm{ \frac{\abs{\xi} \fourier{u_{1}} }{\sqrt{1+\abs{\xi}^2} } \cdot \frac{1+\abs{\xi} }{\abs{\xi} } }_{L^2({\abs{\xi} > 2^\kappa})}
        \leq (1+2^{-\kappa}) \norm{\angles{\grad_{x}}^{-1} \grad_{x} u_{1}}_{L_{x}^2} \leq (1+2^{-\kappa}) \delta_{1}^{1/4} .
    \end{align*}
    The sums can be computed explicitly and we find
    \begin{align*}
        \norm{\ell}_{L_{t}^3 L_{x}^6} \leq 2\frac{2^{\frac{\kappa-1}{3}}}{1-2^{\frac{-1}{3}}} (\norm{u_{0}}_{H_{x}^1} + \norm{u_{1}}_{L_{x}^2}) + 2(1+2^{-\kappa})\delta_{1}^{1/4} .
    \end{align*}
    We may estimate the first sum by $2\sqrt{2J(Q)}$ and choose $\kappa \leq 0$ to minimize the sum above. This happens when 
    \[ 2^{\kappa} = \frac{(2^{1/3}-1)\delta_{1}^{1/4}}{\|u_{0}\|_{H_{x}^1} + \|u_{1}\|_{L_{x}^2}}, \]
    and
    \[ \norm{\ell}_{L_{t}^3 L_{x}^6} \leq C \delta_{1}^{\frac{1}{16}} .\]
    By letting $v$ solve \eqref{eqn:NLKG} with initial data $\Tilde{u}[T_{*}]$ we conclude, thanks to the finite speed of propagation as in \cite[Lemma~2.2]{NakanishiSchlagBook}, that $v$ and $u$ agree on the forward light-cone of base $B_{R}$. Hence $u|_{B_{R_{*}}(0)}[0] = v|_{B_{R_{*}}(0)}[0]$. Moreover, translating $\ell$ to time $T_{*}$, \cite[Lemma~2.19]{NakanishiSchlagBook} provides an absolute constant $C > 0$ such that
    \[ \norm{v[t] - \ell[t]}_{L_{t}^{\infty} H^1 \times L^2} \leq C \delta_{1}^{\frac{1}{16}} ,\]
    assuming $\delta_{1}$ is smaller than an absolute constant. Hence for $t \in [0,T_{*}]$, 
    \[ \norm{\grad_{x} v(t)}_{L^2}
    \leq C \delta_{1}^{\frac{1}{16}} + \norm{\grad \ell(t)}_{L^2}
    \leq C \delta_{1}^{\frac{1}{16}} + \norm{\grad \Tilde{u}(T_*)}_{L^2} + \norm{\LF \Tilde{u}_{t}(T_{*})}_{L^2} \lesssim \delta_{1}^{\frac{1}{16}} .\]
    Therefore
    \begin{align*}
        \norm{u(t)}_{L^2(B_{R_{*}+t}(0))}
        &= \norm{v(t)}_{L^2(B_{R_{*}+t}(0))}
        \leq \norm{v(t)}_{L^6} |B_{R_{*}}(0)+t|^{1/3} \\
        &\lesssim (R_*+t) \norm{\grad v(t)}_{L^2}
        \lesssim (R_*+t) \delta_{1}^{\frac{1}{16}}
    \end{align*}
    and
    \[ \norm{\grad_{x} u(t)}_{L_{x}^2(B_{R_{*}+t})(0)} \leq \norm{\grad_{x} v(t)}_{L^2} \lesssim \delta_{1}^{\frac{1}{16}} .\]
    At this point, letting $\delta_{1} \leq C (\delta_{*} / R_{*})^{16}$ and letting $C > 0$ be small enough in terms of an absolute constant implies that $\norm{u[0]}_{H^1(B_{R_{*}})} <  \delta_{*}/4$, which contradicts the hypotheses.
\end{proof}

We may now transfer this estimate into a lower bound on $K_{2}(\Tilde{u})$ thanks to the variational lower bounds in \cite[Lemma~2.21]{NakanishiSchlagBook}.

\begin{numlemma}
    \label{lemma:LowerBoundK2}
    Let $u$ be a globally defined solution to \eqref{eqn:NLKG} satisfying $E(u, u_{t}) < J(Q)-\delta$, let $\Tilde{u}$ be as in \eqref{eqn:DefinitionTildeu}, and assume the hypotheses of Lemma~\ref{lemma:LowerBoundGradient} are in force. Also assume $|I| = 1$ as before. Then there is $\delta_{2} = \delta_2(\delta, \delta_1) > 0$ such that
    \[ \int_{I} K_{2}(\Tilde{u}(t)) \diff t \geq \delta_{2} .\]
    Furthermore, there is an absolute constant $C >0$ such that $\delta_{2}$ may be chosen equal to $C \delta_{1}\delta$.
\end{numlemma}
\begin{proof}
    From Lemma~2.21 in \cite{NakanishiSchlagBook} there is an absolute constant $c_{0} > 0$ such that
    \[ K_{2}(\Tilde{u}(t)) \geq c_{0} \min \inset{J(Q)-J(\Tilde{u}(t)) , \| \grad_{x} \Tilde{u}(t) \|_{L^2}^2} .\]
    Since $J(\Tilde{u}(t)) \leq E(\Tilde{u}(t)) \leq J(Q) - \delta /2$ by Lemma~\ref{lemma:Localizedu}, we may integrate this bound into
    \[ \int_{I} K_{2}(\Tilde{u}(t)) \diff t
    \geq c_{0}\int_{I} \min \inset{\delta /2, \| \grad_{x} \Tilde{u}(t) \|_{L^2}^2} \diff t.\]
    We define
    \[ J = \inset{t \in I \colon \frac{\delta}{2} < \| \grad_{x} \Tilde{u}(t) \|_{L^2}^2} ,\]
    whence
    \[ \int_{I} K_{2}(\Tilde{u}(t)) \diff t
    \geq  c_{0}\int_{J} \frac{\delta}{2} \diff t
    + c_{0} \int_{I\setminus J} \| \grad_{x} \Tilde{u}(t) \|_{L^2}^2 \diff t . \]
    and consider the following cases. First, if $|J| < \delta_{1} (8J(Q))^{-1}$, where $\delta_{1}$ is as in Lemma~\ref{lemma:LowerBoundGradient}, then 
    \[ \int_{I \setminus J} \norm{\grad_{x}\Tilde{u}(t)}_{L^2}^2 \diff t
    = \int_{I} \norm{\grad_{x}\Tilde{u}(t)}_{L^2}^2 \diff t
    - \int_{J} \norm{\grad_{x}\Tilde{u}(t)}_{L^2}^2 \diff t 
    \geq \int_{I} \norm{\grad_{x}\Tilde{u}(t)}_{L^2}^2 \diff t
    - 4J(Q)|J| \geq \frac{\delta_{1}}{2} .\]
    On the other hand, if $|J| \geq \delta_{1} (8J(Q))^{-1}$ then
    \[ \int_{J} \frac{\delta}{2} \diff t \geq  \frac{\delta \delta_{1}}{16 J(Q)} .\]
    The proof concludes by letting $\delta_{2} = c_{0}\min \inset{(16J(Q))^{-1} \delta \delta_{1} , \delta_{1}/2}$, which is always achieved at the first argument since $\delta < 2J(Q)$. Hence $\delta_{2}$ may be taken explicitly as
    \[ \delta_{2} = \frac{c_{0}}{16J(Q)} \delta \delta_{1} .\]
\end{proof}

We close this section with the proof of an explicit energy channel for $u$ provided that it retains its energy mostly close to the spatial origin at the initial time. 

\begin{numlemma}
    \label{lemma:EnergyChannel}
    Let $u$ be a globally defined solution to \eqref{eqn:NLKG} satisfying $E(u, u_{t}) \leq J(Q)-\delta$, and assume the hypotheses of Lemma~\ref{lemma:LowerBoundGradient} are in force. Moreover, assume (with $R_*$ as in Lemma~\ref{lemma:Bootstrap} and $\delta_2$ as in the preceding lemma)
    \[ E(u[0]|_{B_{R}^c}) < \frac{\delta_{2}}{240 (1+\| \chi \|_{C^2})},\quad R\geq R_* .\]
    Then there are $\delta_{3} >0$ and $T_{*} = T_*(R,\delta_2)>0$ both effectively computable, such that for $T \geq T_{*}$, we have
    \[ \norm{u[T]}_{H^1 \times L^2 (|x| \geq \delta_{3} T )}^2 \geq \frac{\delta_{2}}{16} \]
    Moreover, we may choose $\delta_{3} \leq (8J(Q))^{-1} \delta_{2}$ and
    \[ T_{*} \geq \min \inset{ \frac{3}{2}+R , 8J(Q) \frac{9+4R}{\delta_{2}} -\frac{3}{2} - R } . \]
\end{numlemma}

\begin{proof}
    The proof relies on the localized virial identity
   \begin{equation}
        \label{eqn:Virial}
        \frac{\diff}{\diff t} \langle \Tilde{u}_t , x \cdot \grad_{x} \Tilde{u} + \frac{3}{2}\Tilde{u}\rangle_{L^2} = -K_{2}(\Tilde{u}) + \langle \Tilde{E}, x \cdot \grad_{x} \Tilde{u} + \frac{3}{2} \Tilde{u} \rangle_{L^2} .
    \end{equation}
    Denote $V(t) = \langle \Tilde{u}_t , x \cdot \grad_{x} \Tilde{u} + \frac{3}{2}\Tilde{u}\rangle_{L^2}$. 
    Integrating on $t \in [0,T]$ produces, in view of Lemma~\ref{lemma:LowerBoundK2},
    \begin{align*}
        V(T)-V(0)
        &= -\int_{0}^{T} K_{2}(\Tilde{u}) \diff t + \int_{0}^{T} \langle \Tilde{E}, x \cdot \grad_{x} \Tilde{u} + \frac{3}{2} \Tilde{u} \rangle_{L^2} \diff t \\
        &\leq -T \cdot \delta_{2} + \int_{0}^{T} \langle \Tilde{E}, x \cdot \grad_{x} \Tilde{u} + \frac{3}{2} \Tilde{u} \rangle_{L^2} \diff t .
    \end{align*}
    Using \eqref{eqn:ErrorTerm} on $\langle \Tilde{E}, x\cdot \grad_{x} \Tilde{u} + \frac{3}{2} \Tilde{u} \rangle_{L^2}$ and estimating each resulting term we reach
    \begin{align*}
        \langle \Tilde{E} &, x \cdot \grad_{x} \Tilde{u} + \frac{3}{2} \Tilde{u} \rangle\\
        &\leq \left[2 \|\chi\|_{C^2} + \frac{3}{2(R+t)} + 4\left(2 \|\chi\|_{C^1} + \frac{3}{2(R+t)}\right) + \frac{4CJ(Q)}{R+t} \left(1+\frac{3}{2(R+t)}\right) \right] \\
        &\hspace{5cm}\cdot \left( 1+\frac{3}{2} + \frac{\| \chi \|_{C^1}}{R+t} \right) \int_{R+t \leq |x| \leq 2(R+t)} |u|^2 + |\grad_{x} u|^2 + |u_{t}|^2 \diff x \\
        &\leq 4 \cdot 30(1 + \| \chi \|_{C^2}) E(u[t]|_{R+t < |x|}) \leq 120(1+\| \chi \|_{C^2}) E(u[0]|_{B_{R}^c}) < \delta_{2} / 2,
    \end{align*}
    by Lemma~\ref{lemma:FiniteSpeedOfPropagation} and the choice of $R$ given by Lemma~\ref{lemma:Localizedu}. Hence
    \begin{equation}
        \label{eqn:Virial1}
         - V(T) \geq \frac{T\delta_{2}}{2}  - V(0)
    \end{equation}
    For $\delta_{3} > 0$, we may also estimate
    \begin{equation}
        \label{eqn:Virial2}
        \abs{\langle \charf{\inset{ |x|<\delta_{3} T} } \Tilde{u}_t (T) , x \cdot \grad_{x} \Tilde{u}(T) + \frac{3}{2} \Tilde{u}(T) \rangle_{L^2} }
        \leq 3J(Q) + 2 \delta_{3} T J(Q)
    \end{equation}
    and
    \begin{align*}
        \Big| \langle &\charf{\inset{ |x|<\delta_{3} T} } \Tilde{u}_t (T) , x \cdot \grad_{x} \Tilde{u}(T) + \frac{3}{2} \Tilde{u}(T) \rangle_{L^2} - V(T) \Big| \leq \of{\frac{3}{2} + R + T} \| \Tilde{u}[t] \|^2_{H^1 \times L^2 (|x| \geq \delta_{3} T)} .
    \end{align*}
    By adding $\langle \charf{\inset{ |x|<\delta_{3} T} } \Tilde{u}_t (T) , x \cdot \grad_{x} \Tilde{u}(T) + \frac{3}{2} \Tilde{u}(T) \rangle_{L^2}$ to both sides in \eqref{eqn:Virial1} and using the two estimates above, we reach
    \begin{align*}
        \of{\frac{3}{2} + R + T} \| \Tilde{u}[T] \|^2_{H^1 \times L^2 (|x| \geq \delta_{3} T)} \geq T\of{\frac{\delta_{2}}{2} -2J(Q) \delta_{3}} -V(0) -3J(Q) 
    \end{align*}
    Letting $T_{*} \geq R+3/2$, $\delta_{3} < (8J(Q))^{-1}\delta_2$ and estimating $V(0)$ as in \eqref{eqn:Virial2} results in
    \[ \| \Tilde{u}[T] \|^2_{H^1 \times L^2 (|x| \geq \delta_{3} T)} \geq \frac{\delta_{2}}{8} - J(Q) \frac{9+4R}{\frac{3}{2}+R+T} ,\]
    and we conclude by letting $T_{*} > 8J(Q)(9+4R) \delta_{2}^{-1} - (R+3/2)$. 
\end{proof}

\section{Pointwise decay and the vector fields method}

In this section we develop some of the tools that will be necessary in the proof of Proposition~\ref{prop:ScatteringBound}. In particular, a pointwise decay estimate for scattering solutions is necessary. To achieve this, we will rely on a vector fields approach, particularly following the proof of Proposition~7.3.6 in \cite{Hormander}.

To begin with, we need an energy estimate. Notice that if $v$ is a solution to
\[ (\Box+1) v = f ,\]
then multiplying by $v_{t}$ and integrating on $[t_{0},t]$ yields
\begin{equation}
    \label{eqn:EnergyNonlinear}
    E_{0}(v(t)) = E_{0}(v(t_{0})) + \int_{t_{0}}^{t} \int v_{t}(s) \cdot f \diff x  \diff s, \quad 2E_{0}(v(t)) = \norm{v(t)}_{H^1}^2 + \norm{v_{t}(t)}_{L^2}^2  .
\end{equation}
The boosts $Z_{0j}$ defined by 
\begin{equation}
    \label{eqn:VectorFields}
    Z_{0j} = t\partial_{x_{j}} + x_{j} \partial_{t}
\end{equation}
and spatial derivatives both commute with $(\Box+1)$, which allows us to perform energy estimates on products $Z^I v$ (compositions $Z^I$ of vector fields as in \eqref{eqn:VectorFields}) using \eqref{eqn:EnergyNonlinear}. Coupled with the scattering properties of solutions to \eqref{eqn:NLKG}, this will allow us to deduce time-decay estimates.

For $N \in \N$, an interval $J \subset \R$ and $u$ that solves \eqref{eqn:NLKG}, define
\begin{equation}
    \label{eqn:DefinitionMNSN}
    M_{N}(u(t)) = \sum_{\abs{I} \leq N} E_{0}(Z^{I}u(t)), \quad S_{N}(u) = \sum_{\abs{I} \leq N} \scatnorm{Z^{I}u}, \quad  S_{N}^{J}(u) = \sum_{\abs{I} \leq N} \scnorm{Z^{I}u}{J} .  
\end{equation}
The next lemma shows how scattering solutions preserve the boundedness of the vector fields $Z^{I}$ in the energy and scattering senses.  

\begin{numlemma}
    \label{lemma:VectorFieldsControl}
    Let $N \in \N$, $t_{0} \in \R$ and $u$ be a solution to \eqref{eqn:NLKG} with compactly supported initial data $u[t_{0}] \in H^{N+1} \times H^{N}$ satisfying $\scatnorm{u} < S$. Then there are explicit constants $C_{N}=C_{N}(S) > 0$ and $D_{N}=D_{N}(S) >0 $ such that
    \[ M_{N}(u(t)) \leq D_{N} M_{N}(u(t_{0})), \quad S_{N}(u) \leq C_{N} M_{N}(u(t_{0}))^{1/2} .\]
\end{numlemma}
\begin{proof}
    The strategy relies in appropriately partitioning time in order to exploit Strichartz and energy estimates for $Z^{I} u$ together at each $\ell \leq N$.

    We begin with the case $\ell=1$. Since $\scatnorm{u}<S$, for any $\eps > 0$ we may find $K_{1}=K_{1}(S,\eps)$ and a collection of intervals $\inset{I_{k}}_{k=1}^{K_{1}}$, $I_{k}=[t_{k-1},t_{k})$ such that 
    \[ \scnorm{u}{I_{k}} < \eps , \quad \forall k \in \inset{1,\ldots, K_1} .\]
    Then, for $j \in \inset{1,2,3}$, we have $(\Box+1)Z_{0j}u = 3u^2 Z_{0j}u$, and using Strichartz estimates yields
    \begin{align*}
        \scnorm{Z_{0j}u}{I_{k}} &\leq c\norm{Z_{0j}u[s]}_{H^1 \times L^2} + 3c\eps^2\scnorm{Z_{0j}u}{I_{k}} \\
        &\leq \left(\frac{c}{1-3c\eps^2}\right) \norm{Z_{0j}u[s]}_{H^1 \times L^2} ,
    \end{align*}
    for some $c>0$ and any $s \in I_{k}$.
    Using the energy estimate \eqref{eqn:EnergyNonlinear} on $I_{k}$, we find
    \begin{align*}
        E(Z_{0j}u(t)) &\leq E(Z_{0j}u(t_{k-1})) + 3 \sqrt{2} (\sup_{s \in I_{k}} E(Z_{0j}u(s)))^{1/2} \cdot \scnorm{u}{I_{k}}^2 \scnorm{Z_{0j}u}{I_{k}} \\
        &\leq E(Z_{0j}u(t_{k-1})) + \left(2 \cdot \frac{3c\eps^2}{1-3c\eps^2}\right) (\sup_{s \in I_{k}} E(Z_{0j}u(s))), \quad t \in I_{k},
    \end{align*}
    since $\norm{Z_{0j}u[s]}_{H^1 \times L^2} \leq \sqrt{2} (\sup_{s \in  I_{k}} E(Z_{0j}u(s)))^{1/2}$. Taking the supremum over $t \in I_{k}$ yields
    \[ \sup_{t \in I_{k}} E(Z_{0j}u(t)) \leq \of{1-2 \cdot \frac{3c\eps^2}{1-3c\eps^2}}^{-1}  E(Z_{0j}u(t_{k-1})) .\]
    By the continuity in time, we conclude that
    \[ E(Z_{0j}u(t_{k})) \leq \of{1-2 \cdot \frac{3c\eps^2}{1-3c\eps^2}}^{-1}  E(Z_{0j}u(t_{k-1})) ,\]
    and by arbitrariness of $k$ we find
    \[ E(Z_{0j}u(t)) \leq \of{1-2 \cdot \frac{3c\eps^2}{1-3c\eps^2}}^{-K_1} E(Z_{0j}u(t_{0})), \quad \forall t \in [t_{0},\infty) .\]
    Finally, using the estimate for $\scnorm{Z_{0j}u}{I_{k}}$ at $s=t_{k-1}$,
    \begin{align*}
        \scatnorm{Z_{0j}u} &= \of{\sum_{k=1}^{K_1} \scnorm{Z_{0j}u}{I_{k}}^3}^{1/3}
        \leq \of{\frac{1}{1-3c \eps^2}} \of{\sum_{k=1}^{K_1} \norm{Z_{0j}u[t_{k-1}]}^3_{H^1 \times L^2} }^{1/3} \\
        &\leq \of{\frac{\sqrt{2}}{1-3c \eps^2}} \of{\sum_{k=1}^{K_1} E(Z_{0j} u(t_{k-1}))^{3/2}}^{1/3} \\
        &\leq \of{\frac{\sqrt{2}}{1-3c \eps^2}} E(Z_{0j}u(t_{0}))^{1/2} \of{\sum_{k=0}^{K_1-1} \of{1-2\frac{3c\eps^2}{1-3c\eps^2}}^{-3k/2}}^{1/3} \\
        &\leq C_{1}(S,\eps) \norm{Z_{0j}u[t_{0}]}_{H^1 \times L^2} .
    \end{align*}
    The constant $C_{1}(S,\eps)$ may be computed explicitly to
    \[ C_{1}(S,\eps) =  \of{\frac{1}{1-3c \eps^2}} \of{\frac{ \of{1- 2\frac{3c\eps^2}{1-3c\eps^2}}^{-\frac{3}{2} K_1}  -1}{ \of{1- 2\frac{3c\eps^2}{1-3c\eps^2}}^{-\frac{3}{2}} -1}}^{1/3}, \]
    which is a finite constant that depends on $S$ for each fixed $\eps < \eps_{0}$.
    By fixing $\eps > 0$ small only in terms of the absolute constant $c$, we conclude that there are constants $D_{1}(S), C_{1}(S)$ such that
    \[ M_{1}(u(t)) \leq D_{1}(S) M_{1}(u(t_{0})), \quad S_{1}(u(t)) \leq C_{1}(S) M_{1}(u(t_{0}))^{1/2} .\]

    The case $\ell=\abs{I} > 1$ follows via the same argument, albeit with a few minor modifications that we shall briefly discuss. First, by taking $Z^{I}$ on $(\Box+1)u = u^3$ one finds the Leibniz rule
    \begin{equation}
        \label{eqn:LeibnizOnZI}
        (\Box+1) Z^{I} u = \sum_{I_{1}+I_{2}+I_{3} = I} C(I_{1},I_{2},I_{3}) \cdot Z^{I_{1}} u \cdot Z^{I_{2}} u \cdot Z^{I_{3}} u ,
    \end{equation}
    which introduces a combinatorial constant in the first step that may be estimated by $3^{2\ell}$. Next, one may assume the result for $|I| < \ell$ in order to perform the same estimates as before once a partition $\inset{I_{k}}_{k=1}^{K_{\ell}}$ of $K_{\ell} = K_{\ell}(S_{\ell-1},\eps_{\ell-1})$ intervals is chosen such that $\scnorm{Z^{I}u}{I_{k}} < \eps_{\ell-1}$ for all $k=1, \ldots, K_{\ell}$ and all $|I| < \ell$. Careful bookkeeping of the constants appearing in these steps yields
    \[ M_{\ell}(u(t_{k})) \leq \of{1- \frac{3^{\ell+1}}{2} \cdot \of{\frac{c^2 \ell \eps_{\ell-1}^2}{1-2 \cdot 3^{2\ell}c \eps_{\ell-1}^2}}}^{-k} M_{\ell}(u(t_{0})) \]
    and
    \begin{align*}
        S_{\ell}(u) \leq \of{\frac{c \ell}{1-2 \cdot 3^{2\ell}c \eps_{\ell-1}^2}} M_{\ell}(u(t_{0}))^{1/2} \sum_{k=1}^{K_{\ell}} \of{1- \frac{3^{\ell+1}}{2} \cdot \of{\frac{c^2 \ell \eps_{\ell-1}^2}{1-2 \cdot 3^{2\ell}c \eps_{\ell-1}^2}}}^{-k/2}
    \end{align*}
    for $k=1,\ldots,K_{\ell}$. Again, fixing $\eps_{\ell-1}$ small enough only in terms of $\ell$ and an absolute constant yields the desired constants $C_{\ell} = C_{\ell}(S_{\ell-1}) > 0$ and $D_{\ell} = D_{\ell}(S_{\ell-1}) > 0$ such that
    \[ M_{\ell}(u(t)) \leq D_{\ell} M_{\ell}(u(t_{0})), \quad S_{\ell}(u) \leq C_{\ell} M_{\ell}(u(t_{0}))^{1/2} .\]
\end{proof}

In the proof of the previous lemma we have not explicitly used the regularity assumption on the initial data of $u$. In fact, this is only required for the data of $Z^{I}u$ to belong in $H^1 \times L^2$, as highlighted by the following lemma. 

\begin{numlemma}
    \label{lemma:CombinatorialBound}
    Given $t_{0} > 0$ and a solution $u$ to \eqref{eqn:NLKG} with data $u[t_{0}]$ supported in $B_{R}(0)$, there is an explicit constant $C_{|I|}$ depending only on $|I|$ such that 
    \[ \norm{Z^{I} u[t_{0}]}_{H^1 \times L^2} \leq C_{|I|} t_{0}^{|I|} (1+R^{|I|}) \|u[t_{0}]\|_{H^{|I|+1} \times H^{|I|}} \of{1 + \|u[t_{0}]\|_{H^{|I|+1} \times H^{|I|}}^{|I|+1}} .\] 
\end{numlemma}
\begin{proof}
    By an inductive argument, $Z^{I} u$ is comprised of terms
    \begin{equation}
        \label{eqn:TermsInZI}
        t^{\alpha} x^{\beta} \partial_{x}^{\gamma} \partial_{t}^{\sigma} u , \quad \alpha+|\beta|, |\gamma| + \sigma \leq |I|
    \end{equation}
    Estimating $Z^{I}u$ in $H^1 \times L^2$ by the $H^s$ norms of $u$ and $u_{t}$ requires reducing $\partial_{t}^{\sigma}$ to at most a single time derivative through successive applications of \eqref{eqn:NLKG}. Each of these, however, produces nonlinear terms $u^3$ that give rise to products of derivatives of $u$ and $u_{t}$. This is the main difficulty in estimating $Z^{I}u$.

    Notice that
    \[ \partial_{t}^{\sigma} u
    = \partial_{t}^{\sigma-2} (\Delta u - u + u^3)
    = \Delta \partial_{t}^{\sigma-2} u - \partial_{t}^{\sigma-2} u + \sum_{\sigma_{1} + \sigma_{2} + \sigma_{3} = \sigma-2} \partial_{t}^{\sigma_{1}} u \cdot \partial_{t}^{\sigma_{2}} u \cdot \partial_{t}^{\sigma_{3}} u .\]
    Therefore every two time derivatives produce three terms. One where two time derivatives are traded for two spatial derivatives, a lower order term, and a sum of products of three factors.
    We may perform the same operation for as long as double time derivatives are not exhausted, at least in the first two terms, and by ignoring the products originated by this process, we come to find terms with at most $\sigma$ spatial derivatives and lower order terms. Specifically, the highest order terms in $\partial_{x}^{\gamma} \partial_{t}^{\sigma} u$ will be
    \[ 
    \begin{cases}
        \partial_{x}^{\Tilde{\gamma}} u , \quad &\text{if } \sigma \text{ is even,} \\
        \partial_{x}^{\Tilde{\gamma}} u_{t} , \quad &\text{if } \sigma \text{ is odd,}
    \end{cases}
    \quad \text{where} \quad
    |\Tilde{\gamma}| =
    \begin{cases}
        \abs{\gamma}+\sigma , \quad &\text{if } \sigma \text{ is even,} \\
        \abs{\gamma}+\sigma-1 , \quad &\text{if } \sigma \text{ is odd.} 
    \end{cases}
    \]
    We shall not seek a sharp estimate, and lower order terms will be bounded in $H^1 \times L^2$ by the highest order terms above given the embedding $H^N \embeds H^{s}$ for $s<N$.
    
    It therefore remains to estimate the number of factors coming from the nonlinear interactions. To this end, notice that each application of this rule decreases by two the combined number of derivatives distributed among factors. So we may no longer proceed after every factor contains either none or one time derivative. This means that we can apply the rule at most $\sigma$ times if $\sigma$ is even, and $\sigma-1$ times otherwise, leading to at most $\Tilde{\sigma}$ factors, where
    \[ \Tilde{\sigma} =
    \begin{cases}
        \sigma+1 , \quad &\text{if } \sigma \text{ is even,}\\
        \sigma , \quad &\text{if } \sigma \text{ is odd.}
    \end{cases}
    \]
    The factors appearing in the full expansion of $\partial_{t}^{\sigma} u$ will be of the form $u, u_{t}, \partial_{x}^{p} u$ and $\partial_{x}^{q} u_{t}$, where $p$ can be at most $\sigma$ and $q$ at most $\sigma-1$. We may always estimate them in $H^{1}$ by their $H^{\sigma}$-norms. Furthermore, we can bound the products using the Leibniz rule
    \[ \norm{fg}_{H_{x}^1} \leq \norm{f}_{H^1} \norm{g}_{L^{\infty}} + \norm{g}_{H^1} \norm{f}_{L^{\infty}} \]
    together with the embedding $H^2(\R^3) \embeds L^{\infty}(\R^3)$. 
    
    We may finally estimate $\| Z^{I}u[t_{0}] \|_{H^1 \times L^2}$ as follows.
    First, for $f$ supported on $B_{R}(0)$, we have
    \[ \norm{x^{\beta} f}_{H^1} \leq C_{|\beta|} (1+R^{|\beta|}) \norm{f}_{H^1} .\]
    Therefore
    \[ \norm{t^{\alpha}x^{\beta} \partial_{x}^{\gamma} \partial_{t}^{\sigma} u}_{H_{x}^1}
    \leq 
    C_{|\gamma|} C_{\sigma} C_{|\beta|} t^{\alpha} (1+R^{|\beta|})
    \left[ (\norm{u}_{H_{x}^{|\Tilde{\gamma}|}} + \norm{u_{t}}_{H_{x}^{|\Tilde{\gamma}|-1}} ) + (\norm{u}_{H_{x}^{|\Tilde{\gamma}|}} + \norm{u_{t}}_{H_{x}^{|\Tilde{\gamma}|-1}})^{\Tilde{\sigma}} \right], \]
    where $C_{\sigma}$ is a combinatorial constant that accounts for the number of terms in the full expression of $\partial_{t}^{\sigma}u$ discussed above, $C_{|\gamma|}$ is another combinatorial constant that accounts for the number of terms in the expansion of $\partial_{x}^{\gamma} \partial_{t}^{\sigma} u$, and by all the preceding considerations.

    On the other hand, the expression of the terms in $\partial_{t} Z^{I}u$ will either be of the form $t^{\alpha-1} x^{\beta} \partial_{x}^{\gamma} \partial_{t}^{\sigma} u$, in which case the $L^2$-norm can be bounded using the same strategy as above, or they will contain one more time derivative, i.e. $t^{\alpha} x^{\beta} \partial_{x}^{\gamma} \partial_{t}^{\sigma+1} u$. In this case it suffices to consider different cases according to the parity of $\sigma$.
    \begin{enumerate}
        \item If $\sigma$ is even, then for $\sigma+1$ we find $|\Tilde{\gamma}| = |\gamma| + \sigma$ and at most $\sigma+1$ factors.
        \item If $\sigma$ is odd, then for $\sigma+1$ we find $|\Tilde{\gamma}| = |\gamma|+\sigma+1$ and at most $\sigma+2$ factors.  
    \end{enumerate}
    This yields the estimate
    \[ \norm{t^{\alpha}x^{\beta} \partial_{x}^{\gamma} \partial_{t}^{\sigma+1} u}_{L_{x}^2}
    \leq 
    C_{|\gamma|} C_{\sigma} C_{|\beta|} t^{\alpha} R^{|\beta|}
    \left[ (\norm{u}_{H_{x}^{|\Tilde{\gamma}|}} + \norm{u_{t}}_{H_{x}^{|\Tilde{\gamma}|-1}} ) + (\norm{u}_{H_{x}^{|\Tilde{\gamma}|}} + \norm{u_{t}}_{H_{x}^{|\Tilde{\gamma}|-1}})^{\Tilde{\sigma}} \right], \]
    Notice that we are also enforcing that the constants balance the cases where $\| u \|_{H^{|I|+1} \times H^{|I|}}$ dominates $\| u \|_{H^{|I|+1} \times H^{|I|}}^{\Tilde{\sigma}}$ and vice versa.

    Finally, imposing the aforementioned bounds on $\alpha, |\beta|, |\gamma|$ and $|\sigma|$ yields the desired estimate
    \[ \norm{Z^{I} u[t_{0}]}_{H^1 \times L^2} \leq C_{|I|} t_{0}^{|I|} (1+R^{|I|}) \|u[t_{0}]\|_{H^{|I|+1} \times H^{|I|}} \of{1 + \|u[t_{0}]\|_{H^{|I|+1} \times H^{|I|}}^{|I|+1}} .\]
\end{proof}

\begin{numremark}
    Having explicit bounds for $\| Z^{I}u[t_{0}] \|_{H^1 \times L^2}$ renders the inductive strategy in Lemma~\ref{lemma:VectorFieldsControl} explicitly effective, as the constants $C_{N}$ and $D_{N}$ in Lemma~\ref{lemma:VectorFieldsControl} may be determined in terms of $S$. Indeed, for $N=1$, one may estimate $K_{1} \sim S^3 / \eps^3$ by appropriately choosing $\inset{t_{k}}$ and noting that 
    \[ \scatnorm{u}^3 = \sum_{k=1}^{K_{1}} \scnorm{u}{I_{k}}^3 \sim K_{1} \eps^3 , \quad  K_{1} \sim S^3/\eps^3 .\]
    This makes $C_{1}$ and $D_{1}$ explicit, and using the previous lemma we deduce an explicit bound on the scattering norm of the first iteration. This yields $K_{2}$ and we can proceed inductively.
\end{numremark}

The next step towards deriving the time decay estimates is a suitable adaptation of Proposition 7.3.6 in \cite{Hormander}.

\begin{numprop}
    \label{prop:Hormander}
    Let $n \leq 3$, $u$ satisfy $(\Box+1) u = f$ on $(t,x) \in \R \times \R^n$, $K_{t_{0}} \coloneqq \inset{|x| < |t-t_{0}|}$, for $t_{0}>0$, and $t_{0} < t_{1} < s$. Then for any $(t,x) \in K_{t_{0}}$ with $t_{1}<t<s$, and setting $I_{k}=(2^{k-1}, 2^{k+1})$, 
    \begin{equation}
        \label{eqn:Hormander}
        \sup \limits_{t_{1} \leq t \leq s} t^{n/2}\abs{u(t,x)} \leq
        C \sum_{\abs{I} \leq n} \sum_{k \in \Z} 2^{k} \sup_{I_{k} \cap (t_{1},s)} \| {Z^{I}f(t,\cdot)} \|_{L^2(K_{0})} + \of{\frac{C}{t_{0}}} \sum_{|I| \leq n+2} \sup_{I_{k} \cap (t_{1},s)}  \| {Z^{I} u} \|_{L^2(K_{0})} .
    \end{equation}
\end{numprop}

\begin{proof}
    The proof is identical to the original one in \cite{Hormander}, and we shall only provide a sketch, highlighting the relevant modifications.

    The proof begins by introducing hyperbolic polar coordinates $(\rho, \omega)$, where $\rho = (t^2-\abs{x}^2)^{1/2}$ and the solution $v=\rho^{n/2} u$ to
    \begin{equation}
        \label{eqn:HormanderV}
        \partial_{\rho}^2 v + v = \rho^{n/2} \of{\rho^{-2} \of{\Delta_{H} u + \frac{n(n-2)}{4} u} + f} ,
    \end{equation}
    which we wish to estimate. Here
    \[ \Delta_{H} = \sum_{j=1}^{n} (t \partial_{x_{j}} + x_{j} \partial_{t})^2 - \sum_{1 \leq j < k \leq n}(x_{j} \partial_{x_{k}} - x_{k} \partial_{x_{j}})^2, \]
    is the Laplace-Beltrami operator on the hyperboloid $t^2-|x|^2=1$ for the metric $|\diff x|^2 - |\diff t|^2$. Since we aim to show time decay of order $t^{-3/2}$ for $u$ by applying \eqref{eqn:Hormander} a number of times under increasingly strong assumptions on its time decay, it is necessary to localize $t$ to dyadic scales to make the supremum compatible with the $t$-factor preceding it. We therefore consider $\chi \in C_{0}^{\infty}((1/2,2))$ satisfying
    \[ \sum_{k \in \Z} \chi(t/2^{k}) = 1, \quad t > 0 ,\]
    and
    \[ f_{k}(t,x) = \chi(t/2^k) f(t,x) , \quad u_{k} = \chi(t/2^k)\of{\rho^{-2} \of{\Delta_{H} u + \frac{n(n-2)}{4} u} + f}, \quad 2^k > t_{0} .\]
    The estimate
    \[ |v(\rho)| \leq |v(0)| + |v'(0)| + \int_{0}^{\infty} \abs{ \sigma^{n/2} \of{\sigma^{-2} \of{\Delta_{H} u + \frac{n(n-2)}{4} u} + f}} \diff \sigma , \quad \forall \rho \geq 0 \]
    is immediate from the formula for $v$, and the first two summands on the right-hand side are equal to zero, resulting in bounds for the integrals in the right-hand side of \eqref{eqn:HormanderV} being necessary.

    Using \cite[Lemma~7.3.2]{Hormander}, one concludes that for $t_{1} \leq t \leq s$, $(t,x) \in K_{t_{0}}$, and $\rho^2=t^2-|x|^2$,
    \[ (t/\rho)^{n/2} \int_{0}^{\infty} \sigma^{n/2} \abs{f_{k}(\sigma \omega)} \diff \sigma \leq 2^{k+1}C \sum_{|I| \leq n} \sup_{t_{1} \leq t \leq s} \|{Z^{I} f_{k}(t,\cdot)}\|_{L^2} ,\]
    and
    \[ (t/\rho)^{n/2} \int_{0}^{\infty} \sigma^{\frac{n}{2}-2} |u_{k}(\sigma \omega)| \diff \sigma \leq \of{\frac{C}{t_{0}}} \sum_{|I| \leq n} \sup_{t_{1} \leq t \leq s} \|{Z^{I} u_{k}(t,\cdot)}\|_{L^2} .\]
    Now $|{Z^{I}f_{k}}|$ can be estimated by $\sum_{|J| \leq |I|} |{Z^J f}|$, and since $u$ is radial, $\Delta_{H}$ is made up of boosts $Z^{I}$ for $\abs{I} \leq 2$, which means that $|{Z^{I} u_{k}}|$ may be estimated in terms of $\sum_{|J| \leq |I|+2} |Z^{J}u|$. We finally conclude by performing the sum over $k \in \Z$.
\end{proof}

\begin{numremark}
    It is worth noting that the compact support assumption on the data of $u$ is imposed to make the proof of Lemma~\ref{lemma:CombinatorialBound} and the applications of Proposition~\ref{prop:Hormander} simpler, as one could work in the appropriate weigthed Sobolev spaces where $x^{\beta} \partial_{x}^{\Tilde{\gamma}}$ is bounded. While we shall not make use of this, as compactness of support will be a mild assumption for our purposes, this renders the previous proposition effective when seeking bounds on a light cone for data supported on the entire space.
\end{numremark}

In the sequel we will approximate compactly supported data by means of an approximate identity in order to retain the compactness of the support, but the previous remark suggests that smooth frequency projection can also be used instead.

\begin{numprop}
    \label{prop:LowFrequencyTimeDecay}
    Let $u$ be a solution to \eqref{eqn:NLKG} with initial data $u[t_{1}] \in H^{12} \times H^{11}$ for $0< t_{0} \leq t_{1}$ as before, and assume that $\scatnorm{u} < S$. Then there is a constant $K=K(S,t_{0},t_{1})$ such that
    \[ |u(t,x)| \leq K \cdot t^{-3/2} , \quad t > t_1 .\]
    Moreover, $K$ is explicitly given by
    \[ K = 3^{51} C \max\inset{1,t_{0}^{-1}}^7 D_{11}(S)^{\frac{15}{2}} M_{11}(u(t_{1}))^{\frac{15}{2}} .\]
\end{numprop}

\begin{proof}
    The goal is to show a weaker time decay for $Z^{I}u$ and exploit it to cancel the dyadic powers in the first term on the right-hand side of \eqref{eqn:Hormander}, resulting in a constant. To do this, we express $Z^{I}(u^3)$ as in \eqref{eqn:LeibnizOnZI} and place all factors in $L_{x}^{6}$ in view of Proposition~\ref{prop:Hormander}. After obtaining the weak time decay, we repeat the strategy, this time placing two factors in $L_{x}^{\infty}$ and the remaining one in $L_{x}^2$.

    We shall let $n \leq 3$ for now and impose that $n=3$ later. Proposition~\ref{prop:Hormander} reads 
    \begin{align*}
        \sup_{t_{1} \leq t \leq s} &t^{\frac{n}{2}} |Z^{J}u(t,x)|
        \leq C \sum_{|I| \leq n} \sum_{k \in \Z} 2^k \sup_{I_{k} \cap (t_{1},s)} \|Z^{I}Z^{J}(u^3)\|_{L_{x}^2} \\
        &\hspace{7cm}+ \of{\frac{C}{t_{0}}}\sum_{|I| \leq n+2} \sup_{I_{k} \cap (t_{1},s)} \|Z^{I}Z^{J} u\|_{L_{x}^2} \\
        &\leq C \sum_{|I| \leq n+|J|} \sum_{k \in \Z} 2^k \sup_{I_{k} \cap (t_{1},s)} \sum_{I_{1}+I_{2}+I_{3} = I} C(I_{1},I_{2},I_{3}) \prod_{i=1}^{3} \|Z^{I_{i}}u\|_{L_{x}^6}  \\
        &\hspace{7cm}+ \of{\frac{C}{t_{0}}} \sum_{|I| \leq n+2+|J|} \sup_{I_{k} \cap (t_{1},s)} \| Z^{I}u \|_{L_{x}^2} \\
        &\leq C \sum_{|I| \leq n+|J|} \sum_{k \in \Z} 2^k \sup_{I_{k} \cap (t_{1},s)} \sum_{I_{1}+I_{2}+I_{3} = I} C(I_{1},I_{2},I_{3}) \prod_{i=1}^{3} M_{|I_{i}|}(u(t))^{\frac{1}{2}}  \\
        &\hspace{7cm}+ \of{\frac{C}{t_{0}}} \sum_{|I| \leq n+2+|J|} \sup_{I_{k} \cap (t_{1},s)} M_{|I|}(u(t))^{\frac{1}{2}} \\
        &\leq C \sum_{|I| \leq n+|J|} 3^{|I|} \sum_{k \in \Z} 2^k  \sup_{I_{k} \cap (t_{1},s)} M_{|I|}(u(t))^{\frac{3}{2}}  + \of{\frac{C}{t_{0}}} \sum_{|I| \leq n+2+|J|} \sup_{I_{k} \cap (t_{1},s)} M_{|I|}(u(t))^{\frac{1}{2}} \\
        &\leq C 3^{n+|J|} C_{n+|J|}(S)^{\frac{3}{2}} M_{n+|J|}(u(t_{1}))^{\frac{3}{2}} \sum_{k=\floor{\log_{2} t_{1}}}^{\floor{\log_{2}(s)}} 2^k  \\
        &\hspace{6cm}+ 3^{n+2+|J|} \of{\frac{C}{t_{0}}} C_{n+2+|J|}(S)^{1/2} M_{n+2+|J|} (u(t_1))^{\frac{1}{2}} \\
        &\leq C s 3^{n+|J|} C_{n+|J|}(S)^{\frac{3}{2}} M_{n+|J|}(u(t_{1}))^{\frac{3}{2}} + 3^{n+2+|J|} C_{n+2+|J|}(S)^{1/2} \of{\frac{C}{t_{0}}} M_{n+2+|J|} (u(t_1))^{\frac{1}{2}} .
    \end{align*}
    At $t=s$ on the left-hand side,  this implies that
    \begin{align*}
        |Z^J u(s,x)| \leq C s^{\frac{2-n}{2}} &3^{n+|J|}D_{n+|J|}(S)^{3/2} M_{n+|J|}(u(t_1))^{\frac{3}{2}} \\
        &+ \of{\frac{C}{t_0}} s^{-\frac{n}{2}}  3^{n+2+|J|} D_{n+2+|J|}(S)^{1/2} M_{n+2+|J|} (u(t_1))^{\frac{1}{2}}
    \end{align*}
    for all $s \geq t_1$. For simplicity, let us write
    \[ |Z^J u(s,x)| \leq K_{n+2+|J|}^1 s^{\frac{n-2}{2}}, \]
    where
    \[ K_{n+2+|J|}^1 = K^1_{n+2+|J|} (S,t_0,t_1) = C \max \inset{1,t_0^{-1}} 3^{n+2+|J|}D_{n+2+|J|}(S)^{3/2} M_{n+2+|J|}(u(t_1))^{\frac{3}{2}}. \]
    We come back to \eqref{eqn:Hormander} again, but now we place two vector fields in $L_x^{\infty}$ and one in $L_x^2$. This produces
    \begin{align*}
        \sup_{t_1 \leq t \leq s} &t^{n/2} |Z^J u(t,x)|
        \leq C \sum_{|I| \leq n} \sum_{k \in \Z} 2^k \sup_{I_{k} \cap (t_{1},s)} \|Z^{I}Z^{J}(u^3)\|_{L_{x}^2} + \of{\frac{C}{t_{0}}}\sum_{|I| \leq n+2} \sup_{I_{k} \cap (t_{1},s)} \|Z^{I}Z^{J} u\|_{L_{x}^2} \\
        &\leq C \sum_{|I| \leq n+|J|} \sum_{k = \floor{\log_2(t_1)}}^{\log_2 s} \sum_{I_{1}+I_{2}+I_{3} = I} C(I_{1},I_{2},I_{3}) 2^{(3-n)k} K^1_{n+2+|I_1|} K^1_{n+2+|I_2|} M_{|I_3|}(u(t_1))^{1/2} \\
        &\hspace{4cm}+ \of{\frac{C}{t_{0}}} \sum_{|I| \leq n+2+|J|} \sup_{I_{k} \cap (t_{1},s)} \| Z^{I}u \|_{L_{x}^2} \\
        &\leq C 3^{n+|J|} (K_{2n+2+|J|}^1)^2 M_{n+|J|}(u(t_1))^{1/2} \sum_{k = \floor{\log_2(t_1)}}^{\log_2 s} 2^{(3-n)k} \\
        &\hspace{4cm}+ \of{\frac{C}{t_{0}}} 3^{n+2+|J|} D_{n+2+|J|}(S)^{1/2} M_{n+2+|J|}(u(t_1))^{1/2} .
    \end{align*}
    Once more, looking at $t=s$, and imposing that $n=3$, 
    \begin{align*}
        |Z^J u(t,x)| \leq K_{n+2+|J|}^2 s^{-3/2} \log_2(s),
    \end{align*}
    where
    \begin{align*}
        K_{n+2+|J|}^2 &= K_{n+2+|J|}^2(S,t_0,t_1) \\
        &= C\max\inset{t,t_0^{-1}} 3^{n+2+|J|} (K_{2n+2+|J|}^1)^2 D_{n+2+|J|}(S)^{1/2} M_{n+2+|J|}(u(t_1))^{1/2} .
    \end{align*}
    Finally, we use \eqref{eqn:Hormander} one last time on $u$ rather than $Z^J u$ to conclude. This time we find
    \begin{align*}
        s^{n/2}&|u(s,x)| 
        \leq C \sum_{|I| \leq n} \sum_{k = \floor{\log_2(t_1)}}^{\log_2 s} \sum_{I_{1}+I_{2}+I_{3} = I} C(I_{1},I_{2},I_{3}) 2^{(3-n)k} K^2_{n+2+|I_1|} K^2_{n+2+|I_2|} M_{|I_3|}(u(t_1))^{1/2} \\
        &+ \of{\frac{C}{t_{0}}} \sum_{|I| \leq n+2} \sup_{I_{k} \cap (t_{1},s)} \| Z^{I}u \|_{L_{x}^2} \\
        &\leq C 3^n K_{2n+2}^2 M_{n}(u(t_1))^{1/2} \sum_{k=\floor{\log_2 t_1}}^{\log_2 s} k^2 2^{-2k} + \of{\frac{C}{t_0}} 3^{n+2} D_{n+2}(S)^{1/2} M_{n+2}(u(t_1))^{1/2} .
    \end{align*}
    Now the sum is finite and the remaining terms on the right are constants depending only on $S,t_0,t_1$. We conclude thanks to the factor $s^{n/2}$ on the left-hand side that this is precisely the decay rate we have for $u(t,x)$,
    \[ |u(t,x)| \leq K_{n+2}^3(S,t_0,t_1) \cdot t^{-3/2} .\]
    Finally, we may estimate $K_{n+2}^3$ (recall $n=3$) as follows
    \begin{align*}
        K^3_{n+2}(S,t_{0},t_{1})
        &\leq C 3^{n} \max\inset{1,t_{0}^{-1}} K_{n+2}^{2}(S,t_{0},t_{1})^2 D_{n+2}(S)^{\frac{1}{2}} M_{n+2}(u(t_{1}))^{\frac{1}{2}} \\
        &\leq C 3^{5n} \max\inset{1,t_{0}^{-1}}^3 K_{3n+2}^1(S,t_0,t_1)^4 D_{2n+2}(S)^{\frac{3}{2}} M_{2n+2}(u(t_{1}))^{\frac{3}{2}} \\
        &\leq C 3^{17n} \max\inset{1,t_{0}^{-1}}^7 D_{3n+2}(S)^{\frac{15}{2}} M_{3n+2}(u(t_{1}))^{\frac{15}{2}} , \quad n=3 .
    \end{align*}
\end{proof}

\section{Setting up the proof of the proposition}

We use the previous estimate to provide a bound on the scattering norm of solutions.

\begin{numlemma}
    \label{lemma:PointwiseEstimateAllTime}
    Let $u$ be a scattering solution to \eqref{eqn:NLKG} with smooth data $u[0]$ supported on $B_{R}(0)$ for some $R>0$ and satisfying $\scatnorm{u} < S$. Then there is $K=K(R,S)$ such that
    \[ |u(|t|,x)| \leq K \cdot (1+R+|t|)^{-3/2} , \quad t \neq 0.\]
    Moreover, there is an absolute constant $C>0$ such that
    \begin{equation}
        \label{eqn:ExplicitConstantTimeDecay}
        K(R,S) \leq C D_{11}(S)^{\frac{15}{2}} M_{11}(u(R+1))^{\frac{15}{2}} .
    \end{equation}
\end{numlemma}
\begin{proof}
    Consider $t_{1} = R+1$ and $t_{0}=1$ in Proposition~\ref{prop:LowFrequencyTimeDecay}, and define $v(t)=u(t-t_{1})$. Then $v$ satisfies \eqref{eqn:NLKG} with $v[t_{1}] = u[0]$, and therefore there is $K(R,S)>0$ such that
    \[ |v(t,x)| \leq K(R,S) t^{-\frac{3}{2}} , \quad t>t_{1}\]
    which translates to
    \[ |u(t,x)| \leq K(R,S) (1+R+t)^{-\frac{3}{2}}, \quad t>0 .\]
    The bound for negative times can be deduced from the same argument on $v(t)=u(t_{1}-t)$. 
\end{proof}

\begin{numremark}
    The bound may seem to improve the further we displace the solution in time, but the terms $M_{|I|}$ are not translation-invariant, since the vector fields $Z^{I}$ are not either. Hence the estimates worsen after time translation. Notice as well that the time symmetry yields the same constant $K$ for both positive and negative times.
\end{numremark}

This allows us to prove the following lemma. 

\begin{numlemma}
    \label{lemma:L1L2Estimate}
    Let $\delta > 0$ and assume that $u_{1}, u_{2}$ are scattering solutions to \eqref{eqn:NLKG} with data $u_{1}[0]$ supported on $B_{R_1}$ and $u_2[0]$, respectively, for some $R_{1}>0$, and satisfying $\scatnorm{u_1} < S_{1}$ and $\scatnorm{u_2} < S_2$. Then there is $R_{2} = R_{2}(S_{1}, S_{2},R_1,\delta) > 0$ such that if $u_2[0]$ is supported in $B_{R_2}^c$, then
    \[ \norm{(\Box+1)(u_{1}+u_{2}) - (u_{1}+u_{2})^3}_{L_{t}^1 L_{x}^2} < \delta. \]
\end{numlemma}

\begin{proof}
    We argue by decomposing $u_{1}$ into two parts $u_{1}[0] = u_{1,\ell}[0] + u_{1,h}[0]$.
    To do so, choose an $L^1$--normalized, radial function $\varphi \in C_{c}^{\infty}(B_{1}(0))$ and define the approximate identity $\varphi_{\eps}(x) = \eps^{-n}\varphi(x\eps^{-1})$ for $\eps>0$. Then, set
    \[ u_{1,\ell}[0] = u_{1}[0] \ast \varphi_{\eps} , \quad u_{1,h}[0] = u[0]-u_{1,\ell}[0] .\]
    It is immediate that $u_{1,h}[0] \to 0$ in $H^1 \times L^2$ as $\eps \to 0$, and that both $u_{1,l}[0], u_{1,h}[0]$ are radial. We shall later fix the choice of $\eps>0$ in order to treat $u_{1,h}$ perturbatively and retain enough regularity in $u_{1,\ell}$ to employ the vector fields method and deduce time decay. We will also denote by $R_{1,\ell}$ the radius of the ball over which $u_{1,\ell}[0]$ is supported, $R_{1,\ell} \leq R_1+\eps$.

    We may evolve the data $u_{1,\ell}[0]$ by \eqref{eqn:NLKG} into $u_{1,\ell}$, which means $u_{1,h}$ solves the equation
    \[ (\Box+1) u_{1,h} = u_{1,h}^3 + 3 u_{1,h}^2 u_{1,\ell} + 3 u_{1,h} u_{1,\ell}^2, \]
    evolving the data $u_{1,h}[0]$.
    To recover the scattering hypotheses for these solutions, we treat $u_{1,\ell}$ as a perturbation of $u_{1}$. Indeed, using the frequency localised Strichartz estimate \eqref{eq:Strichartzlocalised} in conjunction with Bernstein's inequality for the linear solution $w_{h}$ with data $u_{1,h}[0]$, we find $\scatnorm{w_{h}} \lesssim \eps^{\frac13}\norm{u_{1,h}[0]}_{H^1 \times L^2}$. Lemma~2.19 in \cite{NakanishiSchlagBook} can be applied to deduce that
    \begin{equation}
        \label{eqn:EstimateHighFrequencies}
        \scatnorm{u_{1,h}} \leq \eps^{\frac13}C(S_{1}) \norm{u_{1,h}[0]}_{H^1 \times L^2}
    \end{equation}
    if $\eps < \eps_{0}(S_{1})$ is small enough. Moreover, we also obtain the estimate
    \[ \scatnorm{u_{1,\ell}} \leq S_{1}+C(S_{1}) \norm{u_{1,h}[0]}_{H^1 \times L^2} = S_{1,\ell} .\]
    The goal now is to estimate the interactions between $u_{1,\ell}$, $u_{1,h}$ and $u_{2}$ in $L_{t}^3 L_{x}^6$, as
    \[ (\Box+1)(u_{1}+u_{2}) - (u_{1}+u_{2})^3 = -3 \of{u_{1}^2u_{2} + u_{1} u_{2}^2}. \]
    In fact, we will estimate $u_{1,h}$ using \eqref{eqn:EstimateHighFrequencies} and take into account the interactions of $u_{1,\ell}$ and $u_{2}$. To this end, we may easily obtain a pointwise estimate for $u_{1,\ell}$ thanks to Lemma~\ref{lemma:PointwiseEstimateAllTime}, which yields $K=K(R_{1,\ell}, S_{1,\ell})$ such that 
    \[ |u_{1,\ell}(|t|,x)| \leq K \cdot (1+R_{1,\ell}+|t|)^{-3/2} .\]
    Letting $\Omega = \supp u_{1,\ell} \cap \supp u_{2}$, and $t_{\Omega} = \inf_{(t,x) \in \Omega \cap \inset{t>0}} t$, we find
    \begin{align*}
        \scatnorm{u_{1,\ell} \charf{\Omega}}^3
        &\leq \int \of{\int \abs{u_{1,\ell}(t,x) \charf{\Omega}(t,x)}^6 \diff x}^{\frac{1}{2}}  \diff t \\
        &\leq \int_{t_{\Omega}}^{\infty} \of{\int \of{K (1+R_{1,\ell} + t)^{-\frac{3}{2}}}^4 \abs{u_{1,\ell}(t,x) \charf{\Omega}(t,x)}^2  \diff x }^{\frac{1}{2}}  \diff t \\
        &\leq K^2 E(u_{1,\ell})^{1/2} \cdot 2\int_{t_{\Omega}}^{\infty} t^{-3} \diff t 
        \leq C K^2 \cdot E(u_{1,\ell})^{1/2} t_{\Omega}^{-2} .
    \end{align*}
    By finite speed of propagation (\cite[Lemma~2.2]{NakanishiSchlagBook}), we conclude that $t_{\Omega} \geq (R_{2}-R_{1,\ell})/2$.
    We can therefore estimate
    \begin{align*}
        \| u_{1}^2u_{2} + u_{1} u_{2}^2 \|_{L_t^1 L_x^2}
        &\leq  \|{u_{1,\ell}^2 u_{2}}\|_{L_t^1 L_x^2} + 2\|{u_{1,\ell} u_{1,h} u_{2}}\|_{L_t^1 L_x^2} \\
        &\hspace{3cm}+ \|{u_{1,h}^2 u_{2}}\|_{L_t^1 L_x^2} + \|{u_{1,\ell} u_{2}^2}\|_{L_t^1 L_x^2} + \|{u_{1,h} u_{2}^2}\|_{L_t^1 L_x^2} \\
        &\leq \left[\of{CK^2E(u_{1,\ell})^{\frac{1}{2}} t_{\Omega}^{-2}}^{\frac{1}{3}}+ \eps^{\frac13}C(S_{1})\norm{u_{1,h}[0]}_{H^1 \times L^2}\right]^2 S_{2} \\
        &\hspace{3cm}+ \left[ \of{CK^2 E(u_{1,\ell})^{\frac{1}{2}}t_{\Omega}^{-2}}^{\frac{1}{3}} + \eps^{\frac13}C(S_{1}) \norm{u_{1,h}[0]}_{H^1 \times L^2} \right] S_{2}^2 .
    \end{align*}
    We may now fix $\eps>0$ so that
    \[ \eps^{\frac23}C(S_{1})^2\|u_{1,h}[0]\|_{H^1 \times L^2}^2 S_{2} + \eps^{\frac13}C(S_{1})\|u_{1,h}[0]\|_{H^1 \times L^2} S_{2}^2 < \delta/12 .\]
    This can be achieved by requiring that
    \begin{equation}
        \label{eqn:ExplicitEpsilon}
        \eps < \Big(\frac{\delta}{12(C(S_{1})^2 S_{2} + C(S_{1}) S_{2}^2)}\Big)^3 .
    \end{equation}
    On the other hand, by choosing $R_{2}$ large enough, we may ensure that 
    \[ \of{CK^2E(u_{1,\ell})^{\frac{1}{2}} t_{\Omega}^{-2}}^{\frac{2}{3}} S_{2} + \of{CK^2E(u_{1,\ell})^{\frac{1}{2}} t_{\Omega}^{-2}}^{\frac{1}{3}} S_{2}^2 < \delta/12 .\]
    This can be achieved by requiring that
    \[ R_{2} > \max \inset{R_{1,\ell}+2, R_{1,\ell}+2 \cdot 12^{\frac{2}{3}} \delta^{-\frac{2}{3}} \left[ \of{C K^2 E(u_{1,\ell})^{\frac{1}{2}}}^{\frac{2}{3}} S_{2} + S_{2}^2 \of{C K^2 E(u_{1,\ell})^{\frac{1}{2}}}^{\frac{1}{3}} \right] } ,\]
    where $K=K(R_{1,\ell}, S_{1,\ell})$ is given explicitly by \eqref{eqn:ExplicitConstantTimeDecay} and the explicit choice of $\eps$ in \eqref{eqn:ExplicitEpsilon}.
    We then conclude that
    \[ \| (\Box+1)(u_{1}+u_{2}) - (u_{1}+u_{2})^3 \|_{L_{t}^1 L_{x}^2} \leq 3\of{2 \frac{\delta}{12} + 2 \frac{\delta}{12}} = \delta .\]
\end{proof}

We may accommodate a small energy piece to the previous solution to be treated perturbatively in order to obtain the following proposition.

\begin{numprop}
    \label{prop:Decoupling}
    Given $S_1, S_2 > 0$, $u_1[0]$ supported on $B_{R_1}(0)$ and $u_2[0]$ with nonlinear evolutions solving \eqref{eqn:NLKG} satisfying $\scatnorm{u_1} < S_1$ and $\scatnorm{u_2} < S_2$, there is $\delta = \delta(S_{1},S_{2}) > 0$ and $R_2 = R_2(S_1,S_2,R_1)$ such that if $u_2[0]$ is supported on $B_{R_2}(0)^c$ and $\|u_3[0]\|_{H^1 \times L^2} < \delta$, then the data
    \[ u[0] = u_1[0] + u_2[0] + u_3[0] \]
    can be globally evolved via \eqref{eqn:NLKG} into $u$ satisfying
    \[ \scatnorm{u} \leq S_{1}+S_{2} + \delta \cdot C_{2}(8/3)^{C (S_{1}+S_{2})^3} < \infty ,\]
    for some absolute constants $C_{2},C > 0$.
\end{numprop}

\begin{proof}
    The argument is a direct application of \cite[Lemma~2.19]{NakanishiSchlagBook}. Denoting by $w_3$ the linear evolution of $u_3[0]$ we may estimate
    \[ \scatnorm{w_3} \leq C \delta \]
    for some $C>0$ using Strichartz estimates \cite[Lemma~2.43]{NakanishiSchlagBook}. On the other hand, we may use Lemma~\ref{lemma:L1L2Estimate} for a given $\delta'>0$ to find $R_2= R_2(S_1,S_2,R_1,\delta') > 0$ such that
    \[ \| \eq(u_1+u_2) \|_{L_t^1 L_x^2} < \delta' , \quad \text{where} \quad \eq(v)=(\Box+1)v-v^3 .\]
    Then, for $0<\eps<\eps_0(S_1+S_2)/2$ from Lemma 2.19 in \cite{NakanishiSchlagBook}, and assuming that $C\delta, \delta'< \eps /2$ , we have the estimate
    \[ \scatnorm{u_3} \leq C_0(S_1+S_2) \eps ,\]
    which results in the estimate
    \[ \scatnorm{u} < S_1+S_2+C_0(S_1+S_2)\eps .\]
    It suffices to fix $\eps$ small enough. For the explicit choice of $\delta$ we refer to the remark below.
\end{proof}

\begin{numremark}\label{rem:growthofradii}
    A careful inspection of the proof of \cite[Lemma~2.19]{NakanishiSchlagBook} can make $\eps_{0}$ and $C_{0}$ explicit. In fact, we find
    \[ \eps_{0}(S)^2 < \frac{C_{1}}{1+\of{\frac{8}{3}}^{CS^3}} , \quad C_{0}(S) \leq C_{2} \of{\frac{8}{3}}^{CS^3}, \]
    for some absolute constants $C,C_{1},C_{2} > 0$. Furthermore, the proofs of Lemma~\ref{lemma:L1L2Estimate} and Lemma~\ref{lemma:PointwiseEstimateAllTime} reveal that 
    \[
    R_2\leq [R_1+1]^{C_3}C_4^{[S_1+S_2]^2}.
    \]
    for absolute positive constants $C_{3,4}$, and the proof of Proposition~\ref{prop:Decoupling} reveals that $\delta$ may be chosen as
    \[ \delta(S_1,S_2)= \frac{\sqrt{C_1}} {C_{3} \sqrt{1+\left(\frac83\right)^{C(S_1+S_2)^3}}}. \]
\end{numremark}

\section{Proof of Proposition~\ref{prop:ScatteringBound}}\label{sec:keyproof}

\begin{proof}[Proof of Proposition~\ref{prop:ScatteringBound}] Using Lemma~\ref{lemma:Bootstrap}, we see that if \eqref{eqn:H1Bootstrap} holds, then the conclusion of the proposition holds with $C_{1}(\delta) = C_*$ an absolute constant. Thus we may assume that there exists a time, which by translation invariance  we may as well set equal to $t = 0$, such that 
\begin{equation}\label{eq:initialtimeenergylocalisation}
\|u[0]\|_{H_{x}^{1}(B_{R_{*}})}  \geq \delta_{*}
\end{equation}
holds, where $R_*, \delta_*$ are as in Lemma~\ref{lemma:Bootstrap}. Now define
    \[ S_{1} = \sup_{E(u) \leq J(Q)-\delta-\gamma\delta_{2}} \scatnorm{u} , \quad S_{2} = \sup_{E(u) \leq J(Q)-\delta-\delta_{*}} \scatnorm{u} \]
where $\gamma: = \frac{1}{240 (1+\| \chi \|_{C^2})}$, $\delta_2$ as in Lemma~\ref{lemma:LowerBoundK2}, and we define the radii $R_j$ as follows, where $\delta_5 = \delta_5(S_1, S_2)$ corresponds to the constant $\delta$ in Proposition~\ref{prop:Decoupling}:
\begin{equation}\label{eq:inductiveradii} R_{0} = \max\{\Tilde{R},R_*\} , \quad R_{j}=R_{2}(S_{1}, S_{2}, R_{j-1}) , \quad j=1, \dots , \delta_{5}^{-1} .\end{equation}
Here $\tilde{R}$ is as in Lemma~\ref{lemma:Localizedu}. Taking advantage of Remark~\ref{rem:growthofradii}, we infer a bound $R_{\delta_5^{-1}}\leq C^{C^{C^{K [S_1+S_2]^3}}}$ for a couple of absolute constants $C,K > 0$.
\\
Next, we consider two cases:
\\

{\it{(i): Concentration of energy near the origin}}. Here we impose the bound 
\[ E(u[0]|_{B_{R_{\delta_5^{-1}}}^c}) < \frac{\delta_{2}}{240 (1+\| \chi \|_{C^2})} . \]
Using Lemma~\ref{lemma:EnergyChannel}, we deduce that there exists $T_* = T_*(R_{\delta_5^{-1}}, \delta_2)$, with the property that 
 \begin{equation}\label{eq:caseienergydispersion}
 \norm{u[t]}_{H^1 \times L^2(|x| \geq \delta_{3} t )}^2 \geq \frac{\delta_{2}}{16}
 \end{equation}
provided $t\geq T_*$. Then set $T_{**}=\max\{\delta_3^{-1}R_{\delta_5^{-1}}, T_*\}$, where $\delta_3$ is as in Lemma~\ref{lemma:EnergyChannel}, and consider the following further dichotomy:
\\

{\it{(i1): Future dispersion of energy}}. Assume the bound 
\begin{equation}\label{eq:futureenergydispersed}
 \|u \|_{L_t^\infty H_{x}^1 ([T_{**},\infty)\times B_{R_*})} < \delta_*,
\end{equation}
where $\delta_*$ is as in Lemma~\ref{lemma:Bootstrap}. In this case we infer from this lemma the bound 
 \[ \|u\|_{L_{t}^3 L_{x}^6([T_{**},\infty)\times \mathbb{R}^3)} \leq  2\big\|u[T_{**}]\big\|_{H^1\times L^2}\leq C_*
 \]
for an absolute constant $C_*$, and further using Sobolev's embedding and Holder's inequality we deduce 
\[
\|u\|_{L_{t}^3 L_{x}^6([0, T_{**}]\times \mathbb{R}^3)} \leq T_{**}^{\frac13}\big\|u\big\|_{L_t^\infty H_x^1}\leq C_*T_{**}^{\frac13}
\]
Taking the sum of the preceding two estimates results in a bound for $\|u\|_{L_{t}^3 L_{x}^6([0, \infty)\times \mathbb{R}^3)}$ of the form asserted in the proposition. 
\\

{\it{(i2): Future concentration of energy}}. Now assume that the estimate \eqref{eq:futureenergydispersed} fails. In that case there exists $T_{***}\geq T_{**}$ such that 
\begin{equation}\label{eq:futureenergyconcentration}
\big\|u(T_{***},\cdot)\big\|_{H_x^1B_{R_*})}\geq \delta_*. 
\end{equation}
Using the pidgeonhole principle, we see that there exists (recall \eqref{eq:inductiveradii}) $R_{j_*}$, $1\leq j_*<\delta_5^{-1}$, and such that if we smoothly truncate the data $u[T_{***}]$ to be supported on a suitable dilate of $B_{R_{j_*}}(0)\backslash B_{R_{j_*-1}}(0)$, the resulting data $u_3[T_{***}]$ satisfies the bound 
\[
\big\|u_3[T_{***}]\big\|_{H^1\times L^2}<\delta_5. 
\]
Furthermore, since by our choice of $R_{**}$ we have $\delta_3T_{***}\geq R_{\delta_5^{-1}}\geq R_{j_*+1}$, we have on account of \eqref{eq:caseienergydispersion} the estimate 
\[
E\big(u[T_{***}]|_{B_{R_{j_*-1}}}\big)<J(Q) - \delta - \gamma\delta_2,
\]
while the assumption \eqref{eq:futureenergyconcentration} yields that 
\[
E\big(u[T_{***}]|_{B_{R_{j_*}}^c}\big)<J(Q) - \delta - \delta_*. 
\]
We conclude that the hypotheses of Proposition~\ref{prop:Decoupling} are satisfied, which in particular results in a bound much better than required for $\big\|u\big\|_{L_t^3 L_x^6}$, thus verifying the proposition globally in this case. 
\\

{\it{(ii): Energy delocalization}}. Returning to our initial time slice, we now make the complementary assumption 
\[ E(u[0]|_{B_{R_{\delta_5^{-1}}}^c}) \geq \frac{\delta_{2}}{240 (1+\| \chi \|_{C^2})}. \]
Also recalling \eqref{eq:initialtimeenergylocalisation}, we can then argue verbatim as in the preceding case {\it{(i2)}} but with $T_{***}$ replaced by $0$ in order to again deduce an estimate which verifies the proposition. 
\\

In case {\it{(i1)}} we only obtained an estimate on $\|u\|_{L_{t}^3 L_{x}^6([0, \infty)\times \mathbb{R}^3)}$, but considering negative times instead (using the invariance of all assumed bounds under time reflection symmetry), the same argument leads to the desired global bound. 
\end{proof}

\section{Above the ground state energy}\label{sec:aboveQ}

Once we move above the energy of the ground state, it becomes impossible to apply the coercivity $K_{0},K_{2} > 0$ and its quantitative counterpart in \cite[Lemma~2.21]{NakanishiSchlagBook}. The variational structure that the problem enjoys, and which ensures that solutions remain a certain distance away from the ground states $\pm Q$, is no longer in force. This requires the careful dynamical analysis present in \cite[Chapter~4]{NakanishiSchlagBook}, which crucially establishes that while a solution may approach the ground state (in the $H^1 \times L^2$--distance), there is a dichotomy in how it achieves so: either this occurs for all forward time beyond a certain point, which ensures the convergence to $\pm Q$, or the solution is \emph{ejected} away from $\pm Q$ and it may never come close to it again.

In the former case, a solution $u$ converging to $\pm Q$ is in fact not a scattering one, and in the latter, the distance between $u$ and $\pm Q$ becomes large enough for all forward time so that the functionals $K_{0}$ and $K_{2}$ stabilize their sign and provide the desired coercivity, and moreover a quantitative lower bound for both of them is in force. The situation then becomes similar to that of solutions below the ground state energy threshold. Scattering will occur when both functionals remain positive, and blow-up happens in the opposite case.
\\

In the following, we shall rely on the description of the dynamics of solutions of \eqref{eqn:NLKG} in the vicinity of $Q$ developed in \cite{NakanishiSchlagBook}. Our goal shall be to quantify the threshold phenomenon in the scattering case, for data 'just below the center stable manifold'. 
\\
To begin with, linearizing the equation around $Q$ results in
\begin{equation}
    \label{eqn:LinearNLKGAroundQ}
    \begin{cases}
        \partial_{t}^2 v + L_{+}(v) = 3Qv^2 + v^3, \\
        u = Q+v , \quad v[0] = u[0]-(Q, 0) ,
    \end{cases}
\end{equation}
where $L_{+} = -\Delta +1 - 3Q^2$. As an operator in $L^2$ acting on radial functions, $L_{+}$ has only one negative eigenvalue $-k^2$, $k>0$, and no eigenvalue at $0$ or in the continuous spectrum $[1,\infty)$ (c.f. \cite[Lemma~3.2]{NakanishiSchlagBook}). Decomposing $v(t)$ in terms of $L_{+}$ then produces
\begin{equation}\label{vdecomposition} v(t) = \lambda(t) \rho + \gamma(t), \quad \gamma \perp \rho ,\end{equation}
satisfying the equations
\begin{equation}
    \label{eqn:LinearNLKGAroundQComponents}
    \begin{cases}
        \partial_{t}^2 \lambda -k^2 \lambda = N_{\rho}(v) \\
        \partial_{t}^2 \gamma + \omega^2 \gamma = P_{\rho}^{\perp} N(v), \quad \omega \coloneqq \sqrt{P_{\rho}^{\perp} L_{+}}, 
    \end{cases}
\end{equation}
and $\rho N_{\rho}(v) \coloneqq P_{\rho} N(v)$ (c.f. \cite[(3.56)]{NakanishiSchlagBook}). The coefficient $\lambda(t)$ drives the stable and unstable modes via the splitting
\[ \lambda_{\pm}(t) \coloneqq \frac{1}{2} (\lambda(t) \pm k^{-1}\lambda'(t)) , \quad \lambda(t) = \lambda_{+}(t) + \lambda_{-}(t) .\]
When in the absence of the unstable mode $\lambda_{+}$, the stable mode $\lambda_{-}$ is responsible for driving solutions towards $\pm Q$ and a free Klein-Gordon wave to which the radiative term $\gamma$ scatters. Otherwise, the unstable mode $\lambda_{+}$ drives the evolution away from the ground states. 
\\
We now have the following description of the threshold manifold separating the scattering from the blow up regime in forward time:
\begin{theorem}\label{thm:centerstable}(\cite[Theorem~3.22]{NakanishiSchlagBook} There exists $\nu>0$ small and a smooth graph $\mathcal{M}$ in the $B_{\nu}(Q, 0)\subset \mathcal{H}_{\text{radial}}$, so that $(Q,0)\in \mathcal{M}$, with tangent plane 
\[
T_Q\mathcal{M} = \{(v_0, v_1)\in \mathcal{H}\quad |\quad \langle kv_0 + v_1, \rho\rangle = 0\},
\]
at $(Q, 0)$, in the sense that 
\[
\sup_{(Q+v_0, v_1)\in B_{\delta}(Q,0)\cap\mathcal{M}}\big|\langle kv_0 + v_1, \rho\rangle\big|\lesssim \delta^2,\,0<\delta<\nu, 
\]
Any data $(u_0, u_1)\in \mathcal{M}$ lead to global in forward time solutions of the form $u = Q + v$ where $v$ satisfies the bounds 
\[
\big\|(v, \dot{v})\big\|_{L_t^\infty\big((0,\infty);\mathcal{H}\big)} + \big\|v\big\|_{L_t^3(0,\infty) L_x^6(\mathbb{R}^3)}\lesssim \nu
\]
and scatters to a free Klein-Gordon solution in $\mathcal{H}$. Writing 
\[
v = \lambda(t)\rho + \gamma(t,\cdot), 
\]
we have the more refined bounds 
\[
\big\|\lambda\big\|_{L_t^1([0,\infty)]}\lesssim \nu,\,\big\|\gamma\big\|_{L_t^2([0,\infty)]\times L_x^6(\mathbb{R}^3)}\lesssim \nu. 
\]
\end{theorem}

We observe that the crucial relation preventing the exponential growth of the parameter $\lambda(t)$ for positive time is the identity 
\begin{equation}\label{eq:codimonecondition}
\dot{\lambda}(0) + k\lambda(0) = -\int_0^\infty e^{-ks}N_{\rho}(v)(s)\,ds. 
\end{equation}
We shall now analyse what happens when we perturb a data pair $\vec{u}(u_0, u_1)\in \mathcal{M}$ in the 'scattering direction'. We will show that the resulting solution forward in time can be quite explicitly described up to a finite 'ejection time', at which the induction on energy scheme from before can be implemented. A key aspect for this shall be the fact that the solution past ejection time admits coercive control on the $K_2$-functional, which shall allow us to deduce an analogue of Lemma~\ref{lemma:LowerBoundK2} even for energies which are above that of the ground state.  

\begin{numlemma}\label{lem:thresholdperturbativepart}There are absolute constants $1\gg \eta_2\gg \eta_1>0$, and such that the following conclusion holds: let $\vec{u}[0]\in B_{\eta_1}(Q,0)\cap \mathcal{M}$, and consider the data 
\[
\vec{\phi}[0]: = \vec{u}[0]-(a\rho,0),\,0<a<\eta_1.
\]
Then the solution of \eqref{eqn:NLKG} with data $\vec{\phi}$ exists on $[0,k^{-1}\log(\frac{\eta_2}{a})]$, and admits a decomposition 
\[
\phi = Q + \mu(t)\rho + \psi
\]
with $\mu(t)\sim -a e^{kt}$, $t\gg \log\big(\frac{\eta_1}{a}\big)$, as well as the error bound 
\[
\big\|(\psi(t,\cdot), \dot{\psi}(t,\cdot))\big\|_{\mathcal{H}}\leq \eta_2 ae^{kt}.
\]
In particular, we have the estimate 
\begin{align*}
\big\|\phi\big\|_{L_t^3([0,T] L_x^6(\mathbb{R}^3)}\sim |\log a|^{\frac13}\|Q\|_{L_x^6}, 
\end{align*}
provided $T = \log(\frac{\eta_2}{a})$. 
\end{numlemma}
\begin{proof} We will represent the unperturbed solution emanating from the data $\vec{u}$ on $\mathcal{M}$ by \eqref{vdecomposition}. Let us write the perturbed solution in the form 
\begin{align*}
\phi(t,\cdot) &= Q(\cdot) + (\lambda(t)+\zeta(t))\rho(\cdot) + (\gamma(t,\cdot) + w(t,\cdot))\\
& =: Q + \Phi. 
\end{align*}
where $Q + \lambda\rho + \gamma = Q+ v$ is the evolution of the data $\vec{u}\in \mathcal{M}$.  Referring to section 3.3.2 in \cite{NakanishiSchlagBook}, we infer that 
\begin{equation}\label{eq:zetaeqn}\begin{split}
\zeta(t) =    -\frac{a}{2}e^{kt}-\frac{a}{2}e^{-kt}  + \frac{e^{kt}}{2k}\int_{0}^t e^{-ks}\delta N_{\rho}(\Phi,v)(s)\,ds 
 + \frac{e^{-kt}}{2k}\int_{0}^t e^{ks} \delta N_{\rho}(\Phi,v)(s)\,ds \\
\end{split}\end{equation}
Here we use the notation 
\[
\delta N_{\rho}(\Phi,v) = N_{\rho}(\Phi) - N_{\rho}(v).
\]
The equation for $w$ is given by 
\begin{equation}\label{eq:wequation}
\partial_{t}^2 w + \omega^2 w = P_{\rho}^{\perp} \delta N(\Phi, v),\,w[0] = 0. 
\end{equation}
where $\delta N(\Phi, v) = N(\Phi) - N(v)$. Now we make the following 
\\

{\bf{Bootstrap assumptions:}}
\[
\big|\zeta(t) + \frac{a}{2}e^{kt}+\frac{a}{2}e^{-kt}\big|\leq \eta_2 a e^{kt}\,\forall t\in [0, T],\,\big\|e^{-kt}w\big\|_{L_t^\infty H_x^1([0,T]\times \mathbb{R}^3)}\leq a\eta_2. 
\]
We show that these bounds improve themselves under the assumptions of the lemma, provided $T\leq k^{-1}\log(\frac{\eta_2}{a})$. 
\\

{\it{Bootstrapping the $\zeta$-bound}}: we can estimate 
\begin{align*}
\Big|\int_{0}^t e^{-ks}\delta N_{\rho}(\Phi,v)(s)\,ds \Big|&\lesssim \int_{0}^t e^{-ks}\langle \rho, |\Phi-v|[|\Phi| + |v|]^2\rangle\,ds\\
& + \int_{0}^t e^{-ks}\langle \rho, |\Phi-v|^3\rangle\,ds.
\end{align*}    
We can bound (using Theorem~\ref{thm:centerstable} and the bootstrap assumptions)
\begin{align*}
\big\|\Phi(s,\cdot)| + |v(s,\cdot)\big\|_{L_s^2([0,t]) L_x^6(\mathbb{R}^3)}\lesssim\eta_2,\,\big\|e^{-ks}|\Phi(s,\cdot)-v(s,\cdot)|\big\|_{L_s^\infty([0,t]) L_x^6(\mathbb{R}^3}\lesssim a,  
\end{align*}
provided $0\leq t\leq T\leq k^{-1}\log(\frac{\eta_2}{a})$.
Using Holder's inequality we then infer that 
\begin{align*}
\int_{0}^t e^{-ks}\langle \rho, |\Phi-v|[|\Phi| + |v|]^2\rangle\,ds + \int_{0}^t e^{-ks}\langle \rho, |\Phi-v|^3\rangle\,ds\lesssim \eta_2^2 a. 
\end{align*}
It follows that 
\begin{align*}
\frac{e^{kt}}{2k}\int_{0}^t e^{-ks}\delta N_{\rho}(\Phi,v)(s)\,ds 
 + \frac{e^{-kt}}{2k}\int_{0}^t e^{ks} \delta N_{\rho}(\Phi,v)(s)\,ds\lesssim a\eta_2^2 e^{kt},
\end{align*}
and hence we improve the first bootstrap assumption, since $\eta_2\ll 1$. 
\\

{\it{Bootstrapping the $w$-bound}}: Here we use the Strichartz estimate in \cite[Lemma~2.46]{NakanishiSchlagBook} to infer the bound 
\begin{align*}
\big\|e^{-kt}w(t,\cdot)\big\|_{H_x^1(\mathbb{R}^3)}\lesssim e^{-kt}\big\|P_{\rho}^{\perp} \delta N(\Phi, v)\big\|_{L_t^1([0,t])L_x^2\mathbb{R}^3)},\,0\leq t\leq T.
\end{align*}
To bound the right hand side, we use that 
\begin{align*}
&e^{-kt}\big\|P_{\rho}^{\perp} \delta N(\Phi, v)\big\|_{L_t^1([0,t])L_x^2\mathbb{R}^3)}\\&\lesssim \big\|\Phi(s,\cdot)| + |v(s,\cdot)\big\|_{L_s^2([0,t]) L_x^6(\mathbb{R}^3)}^2\big\|e^{-ks}|\Phi(s,\cdot)-v(s,\cdot)|\big\|_{L_s^\infty([0,t]) L_x^6(\mathbb{R}^3}\\
&\lesssim a\eta_2^2\ll a\eta_2.  
\end{align*}
The existence of the solution on the asserted interval with the bounds indicated in the boootstrap assumptions is then a consequence of a standard continuity argument. 
The assertion $\mu(t)\sim -a e^{kt}$ follows for $t\gg \log\big(\frac{\eta_1}{a}\big)$ since then
\[
\mu(t) = \lambda(t) + \zeta(t) + \frac{a}{2}e^{kt}+\frac{a}{2}e^{-kt} - [\frac{a}{2}e^{kt}+\frac{a}{2}e^{-kt}],
\]
and we have 
\[
|\lambda(t)| + \big|\zeta(t) + \frac{a}{2}e^{kt}+\frac{a}{2}e^{-kt}\big| + \big|\frac{a}{2}e^{-kt}\big|\ll ae^{kt}. 
\]
The final assertion of the lemma follows by using Holder's inequality, and observing that $\log(\frac{\eta_2}{a})\sim |\log a|$. 
\end{proof}

Combining (\cite[Lemma~4.2]{NakanishiSchlagBook}) together with (\cite[Theorem~4.11]{NakanishiSchlagBook}) and (\cite[Lemma~4.8]{NakanishiSchlagBook}), which are applicable for $\eta_{1}\ll\eta_2\ll1$, we can infer the following corollary:
\begin{numcor}\label{cor:lem:thresholdperturbativepart} Under the preceding assumptions on the $\eta_j$, the solution $\phi(t,\cdot)$ is in the scattering regime for times $t\geq T$, where $T = k^{-1}\log(\frac{\eta_2}{a})$. Furthermore, there is a constant $\eta_3 = \eta_3(\eta_2)>0$ with the property that (here $\kappa_0>0$ is a universal constant)
\[
K_2\big(\phi(t,\cdot)\big)>\min\{\eta_3,\kappa_0\|\nabla_x\phi(t,\cdot)\|_{L_x^2}^2\}\geq 0\,\forall t\geq T. 
\]
\end{numcor}
\begin{proof} Recalling the nonlinear distance function $d_Q(\vec{\phi})$ from \cite{NakanishiSchlagBook}, and furthermore (4.23) there, we can combine the preceding lemma with (\cite[Lemma~4.8]{NakanishiSchlagBook}) in order to deduce that 
\[
K_j\big(\phi(T,\cdot)\big)>\eta_3(\delta_2)>0,\,T = k^{-1}\log(\frac{\eta_2}{a}).\,j = 0, 2. 
\]
Moreover, (\cite[Theorem~4.11]{NakanishiSchlagBook}) in conjunction with (\cite[Lemma~4.8]{NakanishiSchlagBook}) implies that the inequality at the end holds at times $t\geq T$. In particular, the solution is in the scattering regime in forward time. 
\end{proof}

A this point, our strategy shall be to implement the induction on energy step in forward time by following essentially the strategy to establish Proposition~\ref{prop:ScatteringBound}. We require one more technical step allowing us to truncate $\phi$ without violating control over the energy or coercivity of the $K_2$-functional:  
\begin{numlemma}\label{lem:phitrunation} Reset the time $T$ from the preceding lemma to be $0$, and let $\tilde{\phi}_R$ be defined in analogy to \ref{eqn:DefinitionTildeu}, where $\phi$ is as in the preceding lemma. Then under the same assumptions as in the preceding lemma, there exists $\tilde{R} = \tilde{R}(\eta_2)>0$ such that for any $R\geq \tilde{R}$ we have 
\[
E(\tilde{\phi}_R)\leq J(Q) + \delta_*(\eta_{1}),\,K_2\big(\tilde{\phi}_R(t,\cdot)\big)>\min\{\tilde{\eta}_3,\kappa_0\|\nabla_x\tilde{\phi}_R(t,\cdot)\|_{L_x^2}^2\}\geq 0\,\forall t
\]
for suitable $\tilde{\eta}_3 = \tilde{\eta}_3(\eta_2)>0$.     
\end{numlemma}
\begin{proof} In view of (\cite[Lemma~4.8]{NakanishiSchlagBook}) and (4.23) in that same reference, it suffices to show that 
\[
\big\|\tilde{\phi}_R(t,\cdot) - Q\big\|_{H_x^1} + \big\|\partial_t\big(\tilde{\phi}_R(t,\cdot) - Q\big)\big\|_{L_x^2}\geq \delta_*>0\,\forall t\geq 0,
\]
where $\delta_* = \delta_*(\eta_2)>0$. Arguing by contradiction, assume that  
\begin{equation}\label{eq:proximityassumption}
\big\|\tilde{\phi}_R(t,\cdot) - Q\big\|_{H_x^1} + \big\|\partial_t\big(\tilde{\phi}_R(t,\cdot) - Q\big)\big\|_{L_x^2}< \delta_*
\end{equation}
for some $t\geq 0$. The conclusion of the preceding lemma and the no-return theorem (\cite[Theorem~4.11]{NakanishiSchlagBook}) imply that  
\begin{equation}
\big\|\phi(t,\cdot) - Q\big\|_{H_x^1} + \big\|\partial_t\big(\phi(t,\cdot) - Q\big)\big\|_{L_x^2}>\tilde{\delta}(\eta_2)>0. 
\end{equation}
Since $\tilde{\phi}_R = \zeta_R\phi$, see \eqref{eqn:DefinitionTildeu}, we infer that if $\delta_*\ll \tilde{\delta}(\eta_2)$ and $R\geq \tilde{R}(\eta_2)$ is sufficiently large, then necessarily 
\[
\big\|(1-\zeta_R)\phi(t,\cdot)\big\|_{H_x^1}^2 + \big\|\partial_t\big((1-\zeta_R)\big(\phi(t,\cdot)\big)\big)\big\|_{L_x^2}^2\gtrsim\tilde{\delta}^2(\eta_2).
\]
Further, assumption \eqref{eq:proximityassumption} implies that 
\[
E\big(\zeta_R\phi(t,\cdot), \partial_t(\zeta_R\phi(t,\cdot))\big) = J(Q) + O(\delta_*). 
\]
Combining the preceding inequalities and increasing $\tilde{R}(\delta_2)$ if necessary, we infer 
\[
E\big(\phi,\phi_t\big)-J(Q)\gtrsim \tilde{\delta}^2(\eta_2), 
\]
which contradicts the fact that 
\[
E\big(\phi,\phi_t\big)< J(Q) + \delta(\eta_1), 
\]
provided $\eta_2$ is sufficiently large compared to $\eta_{1}$. 
\end{proof}

Armed with the preceding lemma, we can now formulate the step which allows us to infer the bound for the {\it{forward evolution}} of $\phi$ by reducing to the setting of Proposition~\ref{prop:ScatteringBound}.

\begin{numprop}\label{prop:initialdecoupling}There are absolute constants $1\gg\eta_2\gg\eta_1>0$, and such that the following conclusion holds: let $\vec{u}[0]\in B_{\eta_1}(Q,0)\cap \mathcal{M}$, and consider the data 
\[
\vec{\phi}[0]: = \vec{u}[0]-(a\rho,0),\,0<a<\eta_1.
\]
Then the solution of \eqref{eqn:NLKG} with data $\vec{\phi}[0]$ exists globally in forward time. Furthermore, there is a constant $\tilde{\delta}_*(\eta_2)>0$ such that we have 
\[
\big\|\phi\big\|_{L_t^3([T,\infty)]L_x^6(\mathbb{R}^3)}\leq K\big(A\big(J(Q) - \tilde{\delta}_*\big)\big),\, T = k^{-1}\log\big(\frac{\eta_2}{a}\big), 
\]
where for $s<J(Q)$ we set $A(s) = \sup_{E(u) \leq s,\,K_0(u)\geq 0} \scatnorm{u}$, and $K(r) = C^{C^{r^3}}$. 
\end{numprop}
\begin{proof} This follows the same steps as the proof of Proposition~\ref{prop:ScatteringBound}. To begin with, by adjusting $\tilde{R}$ according to Lemma~\ref{lem:phitrunation}, we have the analogue of Lemma~\ref{lemma:LowerBoundGradient}, with $\tilde{u}$ replaced by $\tilde{\phi} = \tilde{\phi}_R$. We also have a direct analogue of Lemma~\ref{lemma:LowerBoundK2} for $\tilde{\phi}$, resulting from Lemma~\ref{lem:phitrunation}, where we now no longer require that the energy be below $J(Q)$, of course. Following the argument in section~\ref{sec:keyproof}, we either directly decompose the solution into two 'pieces' of energy $<J(Q) - \tilde{\delta}_*(\eta_2)$, or else we obtain an a priori bound. We note that the key to being able to 'push the energy' of the individual pieces below $J(Q)$ hinges on our freedom to adjust $\eta_{1}$ downwards, independently of $\eta_2$, which in turn controls the distance to the ground state at ejection time, and hence the coercivity properties of $K_2$ which are responsible for the energy dispersion process. 
    
\end{proof}

We can finally prove Theorem~\ref{thm:main}:
\begin{proof} It suffices to pick $\eta_1>0$ small enough, such that there is $\eta_1\ll\eta_2\ll 1$ such that the conclusions of Proposition~\ref{prop:initialdecoupling} and Lemma~\ref{lem:thresholdperturbativepart} hold, since then, setting $T = k^{-1}\log(\frac{\eta_2}{a})$, 
\[
\big\|\phi\big\|_{L_t^3([0,\infty)L_x^6(\mathbb{R}^3)}\geq \big\|\phi\big\|_{L_t^3([0,T)L_x^6(\mathbb{R}^3)} = \log^{\frac13}(\frac{\eta_2}{a})\frac{\|Q\|_{L_x^6}}{k^{\frac13}} + O(\eta_2)\geq \log^{\frac13}(\frac{1}{a})\frac{\|Q\|_{L_x^6}}{k^{\frac13}} - d_*,
\]
where $d_* = d_*(\eta_{1,2})\rightarrow 0$ as $\eta_{1,2}\rightarrow 0$, which establishes the lower bound of the theorem. For the upper bound, we have 
\begin{align*}
\big\|\phi\big\|_{L_t^3([0,\infty)L_x^6(\mathbb{R}^3)}&\leq \big\|\phi\big\|_{L_t^3([0,T)L_x^6(\mathbb{R}^3)} + \big\|\phi\big\|_{L_t^3([T,\infty)L_x^6(\mathbb{R}^3)}\\
&\leq \log^{\frac13}(\frac{1}{a})\frac{\|Q\|_{L_x^6}}{k^{\frac13}} + O(\eta_2) + K\big(A\big(J(Q) - \tilde{\delta}_*\big)\big).
\end{align*}
The term at the end can be effectively bounded by a finite number of applications of Proposition~\ref{prop:ScatteringBound}. 
\end{proof}

\section{Explicit coercivity above the ground state energy}

Solutions to \eqref{eqn:NLKG} whose energy lies above that of the ground state do not enjoy an \emph{effective} variational estimate as in \cite[Lemma~2.14]{NakanishiSchlagBook}. Specifically, \cite[Lemma~4.8]{NakanishiSchlagBook} provides \emph{noneffective} constants $\eps_0(\delta)$ and $\kappa_1(\delta)$ and an absolute constant $\kappa_0 > 0$ such that
\[ K_0(u) \geq \min(\kappa_1(\delta), \kappa_0 \norm{u}_{H^1}^2) , \quad K_2(u) \geq \min(\kappa_1(\delta) , \kappa_0 \norm{\grad u}_{L^2}^2) \]
whenever $u$ is a scattering solution satisfying
\[ E(u) < J(Q) + \eps_0(\delta)^2 , \quad d_Q(u) \geq \delta .\]
To render our approach fully explicit, it is therefore necessary to make these variational estimates explicit.

\textbf{Up until this point, no AI tools have been used to produce, edit, or improve the arguments present in the paper or their presentation}. In what follows we exhibit an approach to making the aforementioned constants explicit and effective via an argument provided by ChatGPT 6 Astra. We clarify that \textbf{the current section has not been written using AI tools, but the ideas featured throughout have all been produced by AI}.

The proposed strategy hinges on the following three arguments. The first one is a local stability result near $\pm Q$. Assuming that $h \neq 0$ satisfies $J(Q+h)>J(Q)$, $K_0(Q+h)=K_2(Q+h) = 0$ and $\norm{h}_{H^1} \leq r$ for an explicit $r>0$, such a result establishes that
\[ \norm{h}_{H^1}^2 \leq c_0 (J(Q+h)-J(Q)) . \]

The second one is a search procedure to find an explicit $\beta > 0$ for which whenever $f \neq 0$ satisfies $K_0(f) = K_2(f) = 0$ and $J(f) \leq J(Q) + \beta$ then $\dist_{H^1}(f,\pm Q) < r$. This establishes a global stability result for data with energy below $J(Q)+\beta$. The argument relies on approximating $f$ through a finite-dimensional decomposition and ensuring that, for decreasing values of $\beta > 0$, the set of functions that violate the conclusion eventually becomes empty. The finite-dimensional reduction renders the problem computationally feasible, and a compactness argument ensures that the algorithm that needs to be implemented to find $\beta$ eventually stops.

The third step concerns turning the global stability result into a coercivity statement like the one from \cite[Lemma~4.8]{NakanishiSchlagBook}.

\subsection{Local stability near \texorpdfstring{$Q$}{Q}}

In order to understand the static energy $J$ restricted to the constraint $K_j = 0$ near the ground state $Q$, we look at the conditions that perturbations $Q+h$ need to satisfy to remain under the constraint $K_j(Q+h)=0$ (i.e. we look at the tangent space coercivity). To first order, this is $K_j'(Q) = 0$, but
\[ K_j'(Q) h = \angles{L_+ Y_j , h}, \]
where $L_+$ is as in \eqref{eqn:LinearNLKGAroundQ} and
\[ Y_0 = Q , \quad Y_2 = \frac{3}{2} Q + c \cdot \grad Q .\]
Define $\mcalB(f,g) = \angles{L_+ f, g}$. Then
\( Q_{\lambda}(x) = e^{3\lambda/2} Q(e^{\lambda}x) \)
we have (c.f. \cite[Section~2.2]{NakanishiSchlagBook})
\[ K_2(Q_{\lambda}) = e^{2\lambda} \norm{\grad Q}_{L^2}^2 - \frac{3}{4} e^{3\lambda} \norm{Q}_{L^4}^4 ,\]
and differentiating at $\lambda = 0$, together with $K_2(Q)=0$ gives $\nu_2$ in
\[ \nu_0 \coloneqq -\mcalB(Y_0, Y_0) = 2\norm{Q}_{L^4}^4 , \quad \nu_2 \coloneqq -\mcalB(Y_2, Y_2) = \norm{\grad Q}_{L^2}^2 .\]
Therefore $Y_0$ and $Y_2$ are both directions in which the desired coercivity is lost due to the presence of a nonzero component along the unstable eigenmode $\rho$.

The claim now is that whenever $h$ satisfies $\mcalB(Y_j,h) = 0$, then
\[ \mcalB(h,h) \geq c \norm{h}_{H^1}^2 \]
for an explicit, absolute constant $c > 0$.

Indeed, let $q_{\infty} > 0$ be such that $\norm{Q}_{L^{\infty}} \leq q_{\infty}$ and put $\Lambda = 1+3q_{\infty}^2$. Then for any $\gamma \in H^1$,
\begin{equation*}
    \norm{\gamma}_{H^1} = \mcalB(\gamma,\gamma) + 3 \int Q^2 \gamma^2 \leq \Lambda \mcalB(\gamma, \gamma) .
\end{equation*}
Decomposing
\[ Y_j = b_j \rho + z, \quad h = a\rho + \gamma , \quad z,\gamma \perp_{L^2} \rho ,\]
the orthogonality to $Y_j$ reads $k^2 a b_j = \mcalB(z,\gamma)$. Notice that $\mcalB$ is a positive quadratic form on $\rho^{\perp}$ given the gap property of $L_+$ (c.f. \cite[Section~3.1.2]{NakanishiSchlagBook}). Upon an application of Cauchy-Schwarz, $k^2 a b_j = \mcalB(z,\gamma)$ becomes
\[ k^2 a^2 \leq \frac{\mcalB(z,z)}{k^2 b_j^2} \mcalB(\gamma , \gamma) ,\]
which together with $k^2 b_j^2 - \mcalB(z,z) = \nu_j > 0$ yields
\[ \mcalB(h,h) \geq \frac{\nu_j}{k^2 b_j^2} \mcalB(\gamma, \gamma) .\]
Therefore
\[ \norm{h}_{H^1}^2 \leq 2a^2\norm{\rho}_{H^1}^2 + 2 \norm{\gamma}_{H^1}^2 \leq 2\frac{\norm{\rho}_{H^1}^2}{k^2} \mcalB(\gamma, \gamma) + 2 \Lambda \mcalB(\gamma, \gamma) \leq 2 \of{\frac{\norm{\rho}_{H^1}^2}{k^2} + \Lambda} \frac{k^2 b_j^2}{\nu_j} \mcalB(h,h) .\]

This proves the claim. We now transfer this tangent space coercivity to small enough perturbations of the constraint $K_j(Q+h)=0$. Let $S_4>0$ be such that
\[ \norm{h}_{L^4} \leq S_4 \norm{h}_{H^1} \]
and write $q_4 = \norm{Q}_{L^4}$. For $\norm{h}_{H^1} \leq 1$, expanding $K_0, K_2$ gives
\[ K_j(Q+h) = \mcalB(Y_j, h) + R_j(h) , \quad \abs{R_j(h)} \leq R \norm{h}_{H^1}^2 ,\]
where $R$ can be bounded by $R \leq 1+6q_{\infty}^2 + 4q_{4}S_4^3 + s_4^4$ in both cases $j=0,2$. If now $K_j(Q+h) =0$, define
\[ \theta = \frac{\mcalB(Y_j,h)}{\mcalB(Y_j,Y_j)} , \quad w = h-\theta Y_j .\]
Then $\mcalB(Y_j,w) = 0$ and $|\theta| \leq R \nu_{j}^{-1} \norm{h}_{H^1}^2$. Consequently, we have $\norm{w}_{H^1} \geq \norm{h}_{H^1}/2$ provided that
\[ \norm{h}_{H^1} \leq \frac{\nu_j}{2R \norm{Y_j}_{H^1}} ,\]
and
\begin{equation}\label{eqn:AI7}
    \mcalB(h,h) = \mcalB(w,w)-\nu_j \theta^2 \geq \frac{c_j}{4} \norm{h}_{H^1}^2 - \frac{R^2}{\nu_j} \norm{h}_{H^1}^2 .
\end{equation}
Finally,
\[ J(Q+h)-J(Q) = \frac{1}{2} \mcalB(h,h) - \int Q h^3 - \frac{1}{4} \int h^4, \]
and setting
\[ B_3 = q_4 S_4^4 , \quad B_{4,j} = \frac{1}{4} S_4^4 + \frac{R^2}{2 \nu_j} , \quad r_j = \min \inset{1, \frac{\nu_j}{2R \norm{Y_j}_{H^1}}, \frac{c_j}{32 B_3} , \sqrt{\frac{c_j}{32 B_{4,j}}}} ,\]
for
\[ c_j = 2 \of{\frac{\norm{\rho}_{H^1}^2}{k^2} + \Lambda} \frac{k^2 b_j^2}{\nu_j} ,\]
we conclude via \eqref{eqn:AI7} that
\[ J(Q+h) - J(Q) \geq \frac{c_j}{16} \norm{h}_{H^1}^2 \]
whenever $K_j(Q+h) = 0$ and $\norm{h}_{H^1} \leq r_j$.

\subsection{Finite-dimensional approximation}

Set now $r = \min(r_0,r_2)$. The goal is to produce $\beta > 0$ such that whenever $f \neq 0$ satisfies $K_j(f) = 0$ and $J(f) \leq J(Q) + \beta$, then $\dist_{H^1}(f,\pm Q) < r$. This will hinge on a finite-dimensional outer approximation of the set of such functions $f$ that satisfy the first two constraints but not the conclusion, followed by a test that one of these approximations is empty.

\textbf{Initial constraints.} To begin with, notice that we may establish a uniform lower energy bound in order to remove the zero solution from the search. Notice that $K_0(f)=0$ implies that
\[ \norm{f}_{H^1}^2 = 4J(f) \leq 4(J(Q)+1) =: M^2 ,\]
and that $K_2(f)$ yields
\[ \norm{\grad f}_{L^2}^2 = \frac{3}{4} \norm{f}_{L^4}^4 \leq \frac{3}{4} G \norm{f}_{L^2} \norm{\grad f}_{L^2}^3 \leq \frac{3}{4} G M \norm{\grad f}_{L^2}^3, \]
so
\[ \norm{\grad f}_{L^2} \geq \frac{4}{3GM} . \]
Moreover $K_0(f)=0$ yields
\[ \norm{f}_{L^2}^2 + \norm{\grad f}_{L^2}^2 = \norm{f}_{L^4}^4 \leq S_4^4 (\norm{f}_{L^2}^2 + \norm{\grad f}_{L^2}^2)^2 \]
by Sobolev, which gives $\norm{f}_{L^2}^2 + \norm{\grad f}_{L^2}^2 \geq S_4^{-4}$.

\textbf{Defining the approximation operator.} To construct the approximations, we need finite-dimensional subspaces $V_n \subset H_{\rad}^1$ and $H_{\rad}^1$-orthogonal projections $P_n \colon H_{\rad}^1 \longrightarrow V_n$ such that
\[ \norm{f-P_n f}_{L^4} \leq \sigma_n \norm{f}_{H^1} , \quad \sigma_n \to 0 .\]
To this end, Strauss's radial lemma yields
\[ f(r)^2 \leq \frac{1}{4 \pi r^2} \norm{f}_{H^1}^2, \]
which can then be used to ensure that for any $R>1$,
\[ \int_{|x|>R} |f(x)|^4 \diff x \leq \frac{\norm{f}_{H^1}^2}{4 \pi R^2} \int_{|x|>R} |f(x)|^2 \diff x \leq \frac{1}{4\pi R^2} \norm{f}_{H^1}^4 . \]
Letting now
\[
\chi_R(r) =
\begin{cases}
    1, &\quad 0 \leq r \leq R,\\
    2-r/R , &\quad R < r < 2R, \\
    0, &\quad r \geq 2R,
\end{cases}
\]
and setting $g = \chi_R f$, we conclude that $\norm{\grad g}_L^2 \leq \sqrt{2} \norm{f}_{H^1}$ and
\[ \norm{f-g}_{L^4} \leq \frac{1}{(4\pi)^{1/4}} R^{-1/2} \norm{f}_{H^1}. \]
We may therefore restrict our approximations to $B_{2R}$, assuming that $R$ is large enough. And we write $L = 2R$ from now on.

The finite-dimensional decomposition of $f$ is then given in terms of the radial Dirichlet eigenfunctions $e_j^L$ on $B_L$,
\[ e_j^L(x) = \frac{1}{\sqrt{2 \pi L}} \cdot \frac{\sin(j\pi |x|/L)}{|x|} , \quad j \in \inset{1,2,\ldots}, \]
and it extends to zero at every $x \in \R^3 \setminus B_L$. Define the coefficients
\[ \ell_j^R(f) \coloneqq \int_{B_L} \chi_R(x) f(x) e_j^{L}(x) \diff x ,\]
and for $N \geq 1$, the $L^2(B_L)$ spectral truncation of $g$,
\[ F_{R,N} f = \sum_{j=1}^N \ell_j^R(f) e_j^L .\]
The next step is to control the $L^4$ difference of $f$ and $F_{R,N}$ in terms of $R$ and $N$. To this end, write
\[
g=\chi_Rf=\sum_{j=1}^\infty a_j e_j^L,
\qquad
a_j=\ell_j^R(f) .
\]
The Parseval identities for the Dirichlet eigenfunction expansion give
\[
\|g-F_{R,N} f\|_2^2=\sum_{j>N}|a_j|^2, \quad \|\grad (g-F_{R,N}) f\|_2^2=\sum_{j>N} \lambda_j |a_j|^2 \leq \norm{\grad g}_{L^2}^2, \quad \lambda_j = (j \pi / L)^2 .
\]
In particular,
\[ \norm{g-F_{R,N}f}_{L^2} \leq \lambda_{N+1}^{-1/2} \norm{\grad g}_{L^2} . \]
Let $S_6$ be the constant in $\norm{v}_{L^6} \leq S_6 \norm{\grad v}_{L^2}$ for $v \in H^1(\R^3)$ (for instance, a non-optimal choice would be $S_6 = 4$), and interpolate between $L^2$ and $L^6$ to obtain
\[ \norm{g-F_{R,N}f}_{L^4} \leq S_6^{3/4} \lambda_{N+1}^{-1/8} \norm{\grad g}_{L^2} \leq \sqrt{2} S_6^{3/4} \of{\frac{L}{\pi(N+1)}}^{1/4} \norm{f}_{H^1} . \]
Combining this with the cutoff estimate and recalling that $L = 2R$ gives
\[ \norm{f-F_{R,N}f}_{L^4} \leq \sigma(R,N) \norm{f}_{H^1}, \]
where
\[ \sigma(R,N) = (4\pi)^{-1/4} R^{-1/2} + \sqrt{2} S_6^{3/4} \of{\frac{2R}{\pi(N+1)}}^{1/4} .\]

\textbf{Defining the projections.} For each $j$, set $h_j^R = \chi_R e_j^L$. These define the coefficients above,
\[ \ell_j^R(f) = \int f h_j^R ,\]
which are bounded linear functionals on $H^1$, since $\norm{h_j^R} \leq 1$. Therefore they can be represented by a unique function $v_j^R \in H^1$ satisfying
\[ \angles{f,v_j^R}_{H^1} = \ell_j^R(f) , \quad \forall f \in H^1 .\]
In particular we have their explicit description since
\[ v_j^R = (1-\Delta)^{-1} h_j^R .\]
Equivalently,
\[ v_j^R(x) = \int_{R^3} \frac{e^{-|x-y|}}{4 \pi |x-y|} h_j^R(y) \diff y .\]

We can finally define the subspaces $V_n$ as
\[ V_{R,N} = \linspan \inset{v_1^R, \ldots , v_N^R} ,\]
and the projection $P_{R,N}$ to be the $H^1$-orthogonal projection into $V_{R,N}$. We can transfer the $L^4$-estimate from $F_{R,N}$ to $P_{R,N}$ as follows. Let $w = f-P_{R,N}f$. Since $w$ is orthogonal to every $v_j^R$ in $H_{\rad}^1$, we have $\ell_j^R(w) = \angles{w,v_j^R}_{H^1} = 0$, and thus $F_{R,N} w = 0$. Applying the approximation estimate to $w$,
\[ \norm{f-P_{R,N}f}_{L^4} = \norm{w}_{L^4} = \norm{w-F_{R,N}w}_{L^4} \leq \sigma(R,N) \norm{w}_{H^1} \leq \sigma(R,N) \norm{f}_{H^1} .\]
Notice that the orthogonality provided by $P_{R,N}$ is necessary to estimate $\norm{w}_{H^1} \leq \norm{f}_{H^1}$ in the last step.

\textbf{Fixing the error.} Lastly, we can specify the sequence $V_n$ with a prescribed error $\sigma(R,N)$. Letting $\eps_n = 2^{-n}$, $C_{\tail} = (4\pi)^{-1/4}$ and $C_{\inter} = \sqrt{2}S_6^{3/4}$, we may choose
\[ R_n = \ceil{\max \inset{1, \of{\frac{2C_{\tail}}{\eps_n}}^2}} , \qquad N_{n} = \ceil{\max \inset{1, \frac{2R_n}{\pi} \of{\frac{2 C_{\inter}}{\eps_n}}^4}} , \]
so that $\sigma(R_n,N_n) \leq \eps_n / 2$. Therefore
\begin{equation}\label{eqn:AI_B}
   \norm{f-P_nf}_{L^4} \leq 2^{-n} \norm{f}_{H^1} , \quad \norm{P_n f}_{H^1} \leq \norm{f}_{H^1} . 
\end{equation}

\subsection{Testing process}

We now describe how the testing process on possible candidates for a counterexample to a fixed $\beta$ works. A crucial point of the procedure is that it requires certified approximations to $Q$ in $H_{\rad}^1$ with arbitrarily small prescribed error. Having a single approximation with a fixed error might preclude the test from returning a positive result.

In the following, $q_4$ (defined above) can be substituted by any upper bound for $\norm{Q}_{L^4}$. Set $p_n = P_n f$. By the construction of $p_n$, we have
\begin{equation}\label{eqn:AI_CD}
    \abs{\norm{f}_{L^4}^4 - \norm{p_n}_{L^4}^4} \leq e_n \coloneqq 4S_4^3 M^4 2^{-n} , \quad \abs{\int (f-p) Q^3} \leq z_n \coloneqq q_4^3 M 2^{-n} .
\end{equation}
We shall need these estimates later. In the testing process for $p_n$ we shall need energy tests for $\norm{\grad p_n}_{L^2}^2$ and $\norm{p_n}_{L^2}^2$ that we may design as follows. First,
\[ \norm{f}_{L^2}^2 \geq 2 \int f h_{\ell} - \norm{h_{\ell}}_{L^2}^2 .\]
On the other hand, integration by parts yields
\[ \norm{f}_{L^2}^2 \geq 2 \int f (-\triangle h_{\ell}) - \norm{\grad h_{\ell}}_{L^2}^2 .\]
Replacing $f$ by $p_n$ in the linear terms above introduces the errors from \eqref{eqn:AI_B}, which then justifies the constraints
\begin{equation}\label{eqn:AI_G}
    \begin{split}
        \norm{\grad f}_{L^2}^2 &\geq 2 \int p(-\triangle h_{\ell}) - \norm{\grad h_{\ell}}_{L^2}^2 -2^{1-n}MD_{\ell}, \\
        \norm{f}_{L^2}^2 &\geq 2 \int p_n h_{\ell} - \norm{h_{\ell}}_{L^2}^2 - 2^{1-n}M U_{\ell},
    \end{split}
\end{equation}
where the $h_{\ell}$ needs to satisfy $h_{\ell} \in H_{\rad}^1 \cap L^{4/3}$ and $\triangle h_{\ell} \in L^{4/3}$, with the (rational, and not necessarily sharp) upper bounds
\[ \norm{h_{\ell}}_{L^{4/3}} \leq U_{\ell}, \quad \norm{\triangle h_{\ell}}_{L^{4/3}} \leq D_{\ell} .\]
To ensure that the family of test functions $h_{\ell}$ is dense, we may construct them directly from $V_n$. In fact,
\[ \closure{\linspan \bigcup_n V_n} = H_{\rad}^1 ,\]
as if $h \in H_{\rad}^1$ were orthogonal to all the $V_n$ then $P_n h = 0$ and \eqref{eqn:AI_B} gives $\norm{h}_{L^4} = 0$. To ensure that this family is computable, one may enumerate all finite rational linear combinations of the functions $v_j^{R_n}$, $j =1, \ldots , N_n$. This is, we may enumerate
\[ \mcalR_{s} = \inset{\frac{a}{b} \colon |a| \leq s, \ 1 \leq b \leq s} \]
and check, for each $s \leq N_n$, the linear combinations with coefficients in $\mcalR_s$ for all $j=1, \ldots, N_n$ before moving on to $n+1$.

\subsection{The search procedure}

Consider a fixed $n$, and set $\beta_n = 2^{-n}$. For each $k=0,2$, we define $\mcalF_{k,n}$ by the constraints we have previously established. I.e. for $A,B \in \R$ and $p \in V_n$, impose that
\begin{equation}\label{eqn:AI_H}
    A, B \geq 0, \quad A+B \leq M^2 , \quad \norm{p}_{H^1}^2 \leq A+B
\end{equation}
together with
\begin{equation}\label{eqn:AI_G_test}
        A \geq 2 \int p(-\triangle h_{\ell}) - \norm{\grad h_{\ell}}_{L^2}^2 -2^{1-n}MD_{\ell}, \quad 
        B \geq 2 \int p_n h_{\ell} - \norm{h_{\ell}}_{L^2}^2 - 2^{1-n}M U_{\ell},
\end{equation}
as well as
\begin{equation}\label{eqn:AI_IJ}
    \begin{cases}
        A+B \geq S_4^{-4}, &k=0 \\
        A \geq \frac{16}{9 G^2 M^2}, &k=2
    \end{cases}
    \quad
    \frac{1}{2}(A+B) - \frac{1}{4} \norm{p}_{L^4}^4 \leq J(Q) + \beta_n + \frac{1}{4} e_n
\end{equation}
and
\begin{equation}\label{eqn:AI_KL}
    \begin{cases}
        \abs{A+B-\norm{p}_{L^4}^4} \leq e_n &k=0\\
        \abs{A-\frac{3}{4}\norm{p}_{L^4}^4} \leq \frac{3}{4} e_n, &k=2
    \end{cases}
    \quad
    A+B+\norm{Q}_{H^1}^2 -2 \abs{\int pQ^3} \geq r^2-2z_n .
\end{equation}
Suppose now that $f$ satisfies these with $A = \norm{\grad f}_{L^2}^2$, $B = \norm{f}_{L^2}^2$ and $p = P_n f$ against $\beta = 2^{-n}$. Then it constitutes a counterexample to our desired result, and we may move onto $n+1$. If, on the other hand, $\mcalF_{0,n} = \mcalF_{2,n} = \emptyset$, then $\beta = 2^{-n}$ concludes the search.

\subsection{Stopping point for the search}

Fix $k \in \inset{0,2}$. Suppose that $\mcalF_{k,n}$ were nonempty for infinitely many $n$ and choose feasible triples $(p_n,A_n,B_n)$ along a sequence of such $n$.
We remark that for any bounded sequence $u_k \in H_{\rad}^1$ and any fixed $N$, the sequence $P_N u_k$ lies in a bounded subset of $V_N$, and therefore has a convergent subsequence.
Therefore, up to a subsequence,
\[ p_n \to f \quad \text{in} \quad H_{\rad}^1 \cap L^4 \]
and $A_n \to A$, $B_n \to B$. For any fixed $\ell$, the aforementioned test inequalities apply and their errors tend to zero, hence
\[ A \geq 2 \int f (-\triangle h_{\ell}) - \norm{\grad h_{\ell}}_{L^2}^2, \quad B \geq 2 \int f h_{\ell} - \norm{h_{\ell}}_{L^2}^2 . \]
By density of $h_{\ell}$, we conclude that
\begin{equation}\label{eqn:AI_Q}
    A \geq \norm{\grad f}_{L^2}^2 , \quad B \geq \norm{f}_{L^2}^2 .
\end{equation}
Setting $C = \norm{f}_{L^4}^4$ yields
\[ \frac{1}{2} (A+B) - \frac{1}{4} C \leq J(Q), \]
and either $A+B=C$ (if $k=0$) or $A = 3C/4$ (if $k=2$). But by \eqref{eqn:AI_IJ}, $C>0$, so $f \neq 0$.

For $k=0$, choose $\theta$ such that
\[ \theta^2 = \frac{\norm{f}_{L^2}^2 + \norm{\grad f}_{L^2}^2}{\norm{f}_{L^4}^4} \leq 1 ,\]
so $K_0(\theta f) = 0$ and
\[ J(Q) \leq J(\theta f) = \frac{(\norm{f}_{L^2}^2 + \norm{\grad f}_{L^2}^2)^2}{4 \norm{f}_{L^4}^4} \leq \frac{\norm{f}_{L^4}^4}{4} \leq J(Q) ,\]
meaning that $f$ must be a minimizer $f=\pm Q$.
For $k=2$ a similar argument can be run defining $t=\norm{f}_{L^2}^2 / A \in (0,1]$ and computing $K_2(\sqrt{t}f)$ to conclude that $f$ must be $\pm Q$ again, contradicting $r>0$.
This compactness argument ensures that there must be $n_*$ such that $\beta = 2^{-n_*}$.

\subsection{Coercivity estimates}
We now transfer the global stability result into the variational coercivity estimates.
Let
\begin{equation}\label{eqn:AI_25}
    E(f,g) \leq J(Q) + \eps^2 , \quad \dist_{H^1 \times L^2}((f,g),(\pm Q,0)) \geq d .
\end{equation}
Defining
\[ D_0(f) = \norm{f}_{H^1}^2 ,
\quad D_2(f) = \norm{\grad f}_{L^2}^2 ,
\quad a_0 = \frac{1}{2S_4^4},
\quad a_2 = \frac{4}{9G^2 M_*^2} ,\]
where $\norm{f}_{L^2} \leq M_*$,
we first claim that
\[ D_j(f) \leq a_j \implies K_j(f) \geq \frac12 D_j(f) .\]

For $K_0$ this follows from an application of the Sobolev inequality. Taking fourth powers, we find
\[ \norm{f}_{L^4}^4 \leq S_4^4 \norm{f}_{H^1}^4 = S_4^4 D_0(f)^2 , \]
so
\[ K_0(f) \geq D_0(f) - S_4^4 D_0(f)^2 = D_0(f) (1-S_4^4 D_0(f)) \geq \frac12 D_0(f) .\]
For $K_2$, the Gagliardo-Nirenberg inequality yields
\[ \norm{f}_{L^4}^4 \leq G \norm{f}_{L^2} \norm{\grad f}_{L^2}^3 \leq G M_{*} D_2(f)^{3/2} ,\]
so
\[ K_2(f) \ge D_2(f) - \frac34 G M_* D_2(f)^{3/2} = D_2(f) \of{1-\frac34 GM_* \sqrt{D_2(f)}} \geq \frac12 D_2(f) .\]

Consider now the case $D_j(f) \geq a_j$ and $\abs{K_j(f)} \leq \kappa$, with $\kappa \leq a_j/2$. Define
\[ v = \alpha_j f , \quad \alpha_0^2 = \frac{\norm{f}_{H^1}^2}{\norm{f}_{L^4}^4} , \quad \alpha_2^2 = \frac{4\norm{\grad f}_{L^2}^2}{3\norm{f}_{L^4}^4} . \]
Then $K_j(v) = 0$ and
\begin{equation}\label{eqn:AI_27}
    |\alpha_j -1| \leq \frac{2\kappa}{a_j} , \quad \norm{v-f}_{H^1} \leq P_j \kappa , \quad P_j = \frac{2M_*}{a_j} , \quad \norm{v}_{H^1} \leq \sqrt{2}M_* .
\end{equation}
On the ball of radius $2M_*$, a Lipschitz constant for $J$ is
\[ L_J = 2M_* + S_4^4 (2M_*)^3 ,\]
so with $B_j = L_J P_j$,
\[ |J(v)-J(f)| \leq B_j \kappa . \]
If $\eps^2 + B_j \kappa \leq \beta$, then the stability from the first part gives
\begin{equation}\label{eqn:AI_29}
    \dist_{H^1} (v,\pm Q)^2 \leq C (\eps^2 + B_j \kappa) .
\end{equation}
Furthermore, $J(v) \geq J(Q)$ implies
\[ \norm{g}_{L^2}^2 \leq 2 (\eps^2 + B_j \kappa), \]
which together with \eqref{eqn:AI_27} yields
\[ \dist_{H^1 \times L^2}((f,g) , (\pm Q,0))^2 \leq D_j^{*} (\eps^2 + \kappa) ,\]
where $D_j^* = 2P_j^2 + 2(C+1)(1+B_j)$ suffices when $\kappa \leq 1$. Set now
\[ a = \min(a_0,a_2), \quad B_* = \max(B_0,B_2) , \quad D_* = \max(D_0^*, D_2^*) ,\]
and choose
\[ \eps_{\var}(d)^2 = \min \inset{1, \frac{\beta}{4}, \frac{d^2}{4D_*}} , \quad \kappa_{\var}(d) = \min \inset{1, \frac{a}{2} , \frac{\beta}{4D_*}, \frac{d^2}{4D_*}} .\]
Then \eqref{eqn:AI_29} contradicts \eqref{eqn:AI_25} whenever $D_j(f) \geq a_j$ and $|K_j(f)| \leq \kappa_{\var}(d)$. Therefore we conclude that
\[ K_0(f) \geq \min \inset{\kappa_{\var}(d) , \frac12 \norm{f}_{H^1}^2} , \quad K_2(f) \geq \min \inset{\kappa_{\var}(d) , \frac12 \norm{\grad f}_{L^2}^2} .\]

% \nocite{*}                                 % Include all the resources
\printbibliography                          % References and toc inclusion
\end{document}

%% file: files/preamble.tex
\usepackage[english]{babel}             % English
\usepackage{geometry}
\usepackage{amsfonts,amsmath,amssymb,amsthm}    % AMS maths packages
\usepackage{mathtools}
\usepackage{bbm}
\usepackage{esint}
\usepackage{lmodern}                    % For font management
\usepackage{graphicx}                   % For images and figures
\usepackage{xcolor}                     % For using colors
\usepackage[backend=biber,style=alphabetic,bibencoding=utf8,isbn=false,url=false,sorting=nty,doi=false,giveninits=true,maxbibnames=10,maxalphanames=10]{biblatex}  % For the bibliography
\AtEveryBibitem{\clearfield{month}}
\usepackage{csquotes}                   % So that overleaf doesn't show warnings
\usepackage{enumitem}                   % Enumerations with roman indices
\usepackage{parskip}
\usepackage[margin=1cm,labelfont=bf]{caption}   % For custom caption format 
\usepackage{hyperref}                   % For the links
\usepackage{float}

\savegeometry{paper}                    % Saves the geometry as 'paper'

\graphicspath{{./files/figures/}}       % Path for figures

\AtBeginBibliography{\small}            % Small font

\allowdisplaybreaks                     % Breaks in math mode

\setlist[enumerate,1]{label=(\roman*)}  % Small roman in enumerations

\numberwithin{equation}{section}

\DeclareFieldFormat{extraalpha}{#1}     % Format for the citations in the text
\DeclareLabelalphaTemplate{
  \labelelement{
    \field[final]{shorthand}
    \field{label}
    \field[strwidth=2 , strside=left,ifnames=1]{labelname}
    \field[strside=left,ifnames=2-,varwidthlist=true]{labelname}
  }
}
\DeclareLabelalphaNameTemplate{
  \namepart[use=true, pre=true, strwidth=1, compound]{prefix}
  \namepart{family} % Compounds are a pain in the ass
}

\renewbibmacro{in:}{}                   % Suppresses the 'In:' before journal
\DeclareFieldFormat[article,periodical]{volume}{\mkbibbold{#1}}
\DeclareFieldFormat[article,periodical]{journaltitle}{#1\isdot}
\DeclareFieldFormat[article,periodical]{title}{#1}

%% file: files/macros.tex
\newcommand{\mcalB}{\mathcal{B}}

\newcommand{\mcalF}{\mathcal{F}}

\newcommand{\mcalR}{\mathcal{R}}

\newcommand{\naturals}{\mathbb{N}}
\newcommand{\N}{\naturals}
\newcommand{\integers}{\mathbb{Z}}
\newcommand{\Z}{\integers}

\newcommand{\reals}{\mathbb{R}}
\newcommand{\R}{\reals}

\newcommand{\field}{\mathbb{F}}

\newcommand{\eps}{\varepsilon}
\newcommand{\fourier}[1]{\widehat{#1}}

\newcommand{\norm}[1]{\left\|  #1  \right\|}
\newcommand{\abs}[1]{\left| #1 \right|}

\newcommand{\of}[1]{\left( #1 \right)}

\newcommand{\inset}[1]{\left \{ #1 \right \}}
\newcommand{\angles}[1]{\left\langle #1 \right\rangle}

\newcommand{\closure}[1]{\overline{#1}}

\newcommand{\grad}{\nabla}

\newcommand{\charf}[1]{\mathbbm{1}_{#1}}

\DeclareMathOperator{\supp}{supp}

\DeclareMathOperator{\diff}{d\!}

\DeclareMathOperator{\divergence}{div}

%% file: files/environments.tex
\newcommand{\theoremname}{Theorem}
\newcommand{\lemmaname}{Lemma}
\newcommand{\propositionname}{Proposition}
\newcommand{\definitionname}{Definition}
\newcommand{\corollaryname}{Corollary}
\newcommand{\claimname}{Claim}

\newcommand{\examplename}{Example}
\newcommand{\remarkname}{Remark}
\newcommand{\commentsname}{Comments}
\newcommand{\notename}{Note}
\newcommand{\notationname}{Notation}

\newcounter{allcounter}[section]

\newtheorem{numthm}[allcounter]{\theoremname}
\newtheorem{theorem}{\theoremname}

\newtheorem{numlemma}[allcounter]{\lemmaname}

\newtheorem{numprop}[allcounter]{\propositionname}

\newtheorem{numcor}[allcounter]{\corollaryname}

\newtheorem*{remark}{\remarkname}                   % Remark
\newtheorem{numremark}[allcounter]{\remarkname}